\documentclass[reqno]{amsart}
\usepackage{color}
\usepackage{amsmath,amssymb,amsfonts,amsthm,bm}
\usepackage{mathrsfs}
\usepackage{graphicx}
\usepackage{enumerate}
\usepackage[numbers, sort&compress]{natbib}
\usepackage[shortlabels]{enumitem}
\usepackage[paperwidth=157mm,paperheight=235mm,tmargin=20mm,lmargin=18mm,rmargin=18mm,bmargin=22mm,asymmetric]{geometry}

\newtheorem{theorem}{Theorem}[section]
\newtheorem{corollary}[theorem]{Corollary}
\newtheorem{lemma}[theorem]{Lemma}
\newtheorem{proposition}[theorem]{Proposition}
\theoremstyle{assumption}

\theoremstyle{definition}
\newtheorem{definition}[theorem]{Definition}
\theoremstyle{remark}
\newtheorem{remark}[theorem]{Remark}
\numberwithin{equation}{section}

\newcommand{\eps}{\varepsilon}
\newcommand{\norm}[1]{\Vert#1\Vert}
\newcommand{\abs}[1]{\left\vert#1\right\vert}
\newcommand{\set}[1]{\left\{\,#1\,\right\}}
\newcommand{\inner}[1]{\left(#1\right)}
\newcommand{\comi}[1]{\left<#1\right>}

\newcommand{\normm}[1]{{ \vert\kern-0.25ex \vert\kern-0.25ex \vert #1
		\vert\kern-0.25ex \vert\kern-0.25ex \vert}}
\def\eqdef{\buildrel\hbox{\footnotesize def}\over =}

\makeatletter

 \newbox \abstractbox
\renewenvironment{abstract}{\global\setbox\abstractbox=\vbox\bgroup
 \hsize=\textwidth
  \vskip 1.2cm
  \noindent\unskip \textbf{Abstract.}
 }
 {
 \egroup}

\@namedef{subjclassname@2020}{%
	\textup{2020} Mathematics Subject Classification}

 \def\@startsection#1#2#3#4#5#6{%
 \if@noskipsec \leavevmode \fi
 \par \@tempskipa #4\relax
 \@afterindentfalse
 \ifdim \@tempskipa <\z@ \@tempskipa -\@tempskipa \@afterindentfalse\fi
 \if@nobreak \everypar{}\else
     \addpenalty\@secpenalty\addvspace\@tempskipa\fi
 \@ifstar{\@dblarg{\@sect{#1}{\@m}{#3}{#4}{#5}{#6}}}%
         {\@dblarg{\@sect{#1}{#2}{#3}{#4}{#5}{#6}}}%
}

\def\@settitle{%
  \bgroup
  \centering
  \vglue1cm
  \fontsize{12}{15}\fontseries{b}\selectfont
  \@title
  \vskip 20pt plus 6pt minus 8pt
  \egroup
}

\def\@setauthors{%
  \begingroup
  \trivlist
  \centering \bfseries
 \normalsize\@topsep30\p@\relax
  \advance\@topsep by -\baselineskip
  \item\relax
  \andify\authors
 {\rmfamily\authors}%
  \endtrivlist
  \endgroup
}

\def\@setaddresses{\par
  \nobreak \begingroup
\normalsize
  \def\author##1{\nobreak\addvspace\bigskipamount}%
  \def\\{\unskip, \ignorespaces}%
  \interlinepenalty\@M
  \def\address##1##2{\begingroup
    \par\addvspace\bigskipamount\noindent
    \@ifnotempty{##1}{(\ignorespaces##1\unskip) }%
    {\ignorespaces##2}\par\endgroup}%
  \def\curraddr##1##2{\begingroup
    \@ifnotempty{##2}{\nobreak\indent{\itshape Current address}%
      \@ifnotempty{##1}{, \ignorespaces##1\unskip}\/:\space
      ##2\par}\endgroup}%
  \def\email##1##2{\begingroup
    \@ifnotempty{##2}{\nobreak\noindent{\itshape E-mail address}%
      \@ifnotempty{##1}{, \ignorespaces##1\unskip}\/:
       ##2\par}\endgroup}%
   \def\urladdr##1##2{\begingroup
    \@ifnotempty{##2}{\nobreak\indent{\itshape URL}%
      \@ifnotempty{##1}{, \ignorespaces##1\unskip}\/:\space
      \ttfamily##2\par}\endgroup}%
  \addresses
  \endgroup
}

 \renewcommand\section{\@startsection{section}{1}{\z@}%
{27pt plus 6pt minus 8pt}{14pt plus 6pt minus 8pt}
{\center\normalfont\large\bfseries}}

\makeatother

\begin{document}

\title[Hyperbolic quasi-linear Prandtl equations]{Gevrey   well-posedness of quasi-linear hyperbolic Prandtl equations}

\author[ W.-X. Li, T. Yang \and P. Zhang]{Wei-Xi Li, Tong Yang \and  Ping Zhang}

\date{}

\address[W.-X. Li]{School of Mathematics and Statistics,   Wuhan University,  Wuhan 430072, China
	\& Hubei Key Laboratory of Computational Science, Wuhan University, Wuhan 430072,  China}

\email{wei-xi.li@whu.edu.cn}

\address[T.Yang]{PolyU}

\email{t.yang@polyu.edu.hk}

\address[P. Zhang]
	{Academy of
		Mathematics $\&$ Systems Science and  Hua Loo-Keng Key Laboratory of
		Mathematics, The Chinese Academy of Sciences, China, and School of Mathematical Sciences, University of Chinese Academy of Sciences, Beijing 100049, China.}
	\email{zp@amss.ac.cn}


\begin{abstract}
We study the hyperbolic version of the Prandtl system derived from the hyperbolic Navier-Stokes system with no-slip boundary condition. Compared to the classical  Prandtl system, the quasi-linear terms in the hyperbolic Prandtl equation leads to an additional instability
	mechanism. To overcome the loss of derivatives in all directions in the quasi-linear term, we introduce a new  auxiliary function for the well-posedness of the system in an anisotropic Gevrey space which is Gevrey class $\frac 32$ in the tangential variable and is analytic  in the normal variable.
\end{abstract}

 \maketitle


\section{Introduction}

We investigate  the  well-podedness of the following quasi-linear hyperbolic Prandtl system in the half-space $\mathbb  R_+^d \eqdef\big\{(x,y); \ x\in\mathbb R^{d-1}, y>0\big \}  $  with $d=2$ or $3:$
\begin{equation}\label{hyperPr}
\left\{
\begin{aligned}
& \eta\partial_t^2u+\partial_tu+(u\cdot\partial_x) u+v\partial_yu+\eta\partial_t\big(  (u\cdot \partial_x )u+v\partial_y u\big) - \partial_y^2u+ \partial_x p=0,  \\
&\partial_x u +\partial_yv =0,\\
&u|_{y=0}=v|_{y=0}=0,\quad u|_{y\rightarrow+\infty}=U,\\
&u|_{t=0} =u_0,\quad \partial_t u|_{t=0}=u_1,
\end{aligned}\right.
\end{equation}
where $0<\eta<1$ is a small parameter. The unknown $u$  represents the tangential velocity which is   scalar in the two-dimensional (2D) case  and vector-valued in 3D. And the functions  $p=p(t,x)$ and $U=U(t,x)$ in \eqref{hyperPr} are the traces of the tangential velocity field and pressure of the outer flow on the boundary satisfying that
\begin{equation*}
	 \eta\partial_t^2 U+\partial_tU+U\cdot\partial_x U +\eta\partial_t\big(  U\cdot \partial_x U\big) + \partial_x p=0.
\end{equation*}
 This  degenerate  hyperbolic system  \eqref{hyperPr} can be
 derived  from the  hyperbolic Navier-Stokes equations  with  the  no-slip boundary condition.   It is well-known
 that the classical  Navier-Stokes system can be obtained from the Newtonian  law. And its parabolic 
 structure leads to  the property of infinite speed of propagation which seems to be a paradox from the physical point of view.   To have finite propagation,  Cattaneo \cite{MR0032898, MR95680}   proposed  to  replace the Fourier law by the  so-called  Cattaneo law,  where a small time delay $\eta$ is introduced  in  stress tensors. And this yields the following hyperbolic version of  Navier-Stokes equations:
 \begin{equation}\label{hyns}
 	 \eta\partial_t^2u^{NS}+\partial_tu^{NS}+(u^{NS}\cdot\nabla) u^{NS} +\eta\partial_t\big(  (u^{NS}\cdot\nabla) u^{NS}\big) - \eps \Delta  u^{NS}+ \nabla p ^{NS} =0,
 \end{equation}
 where the gradient operator  $\nabla$ is taken with respect to all spatial variables,  and similarly for the Laplace operator $\Delta$.   In the whole space, the system \eqref{hyns} with fixed viscosity $\eps>0$ was studied by Coulaud-Hachicha-Raugel \cite{MR4506781}  in  almost optimal function spaces (see also \cite{MR3085226, MR3942552}).  On the other hand, it is  natural to study the inviscid limit of \eqref{hyns} as $\eps\rightarrow 0$, in particular in  the situation when the fluid domain has a physical boundary. In fact,   when we analyze  the  asymptotic expansion with respect to  the viscosity  $\eps$ of  \eqref{hyns} with the  no-slip boundary condition,   a Prandlt type boundary layer is expected to take care of the mismatched   tangential velocities. In fact, 
  the governing equation of the boundary layer  is  the system \eqref{hyperPr}  by following the Prandtl's ansatz.

When $\eta=0$, the system \eqref{hyperPr}  is the classical Prandtl equations. The mathematical study of  the classical Prandtl boundary layer has a  long history with fruitful results and developed approaches in analysis. It has been well studied  in various function spaces,  see, e.g., \cite{MR3327535, MR3795028, MR3670620,  MR3983729, MR3458159, MR3600083, MR3925144,  MR1476316, MR2601044, MR3429469, MR2849481,  MR3461362, MR3284569,     MR3493958, MR4055987,  MR2020656,MR3464051,CMAA-1-345,MR4271962,2021arXiv210300681Y} and the references therein. Due to the loss of tangential derivatives in the nonlocal term $v\partial_yu$, the Prandtl system is usually ill-posed in Sobolev spaces.  It is now well understood that, for initial data without any structural assumption,  the  Prandtl system is well-posed in Gevrey class with optimal Gevrey index $ 2$ by the instability
 analysis of    G\'erard-Varet and Dormy \cite{MR2601044} and the work on well-posedness of
 Dietert-G\'erard-Varet \cite{MR3925144} and  Li-Masmoudi-Yang \cite{MR4465902}.   The key observation in \cite{MR3925144, MR4465902} is about  some kind of  intrinsic  structure  that is similar to hyperbolic feature for one order loss of tangential derivatives. Recently, inspired by the stabilizing effect  of the intrinsic hyperbolic type structure, Li-Xu-Yang  \cite{MR4494626} showed  the global well-posedness of  a Prandtl Model from MHD in the Gevrey 2 setting.

 The hyperbolic Prandtl system \eqref{hyperPr} 
 is more complicated in terms of loss of derivatives due to the  quasi-linear term  $\eta\partial_t(u\cdot \partial_xu+v\partial_yu)$.  In fact, as to be seen below, the loss of derivatives occurs     not only for the tangential variable  but also for the normal variable.    Inspired by the abstract Cauchy-Kowalewski theory, one can expect the well-posedness of \eqref{hyperPr} in the full analytic spaces (i.e., space of functions that are analytic in all variables).    However,  it is hard to relax the analyticity to Gevrey class because the nonlinearity and non-locality in the term $\partial_t((u\cdot \partial_x)u+v\partial_yu)$ that prevent us to apply the techniques developed for
   the classical Prandtl equation directly. An attempt is to consider the following semi-linear model
 \begin{equation}\label{semipr}
 	 \eta\partial_t^2u+\partial_tu+(u\cdot\partial_x) u+v\partial_yu- \partial_y^2u+ \partial_x p=0
 \end{equation}
 by   removing  the  quasi-linear term $\eta\partial_t\big(  (u\cdot \partial_x) u+v\partial_y u\big) $ in \eqref{hyperPr}.  For this, the local well-posedness of \eqref{semipr} in Gevrey 2 space is obtained by \cite{MR4498949} which is   the same Gevrey space  for the classical Prandtl equation.
 However,  the Gevrey  index 2 may not be  optimal for the well-posedness of \eqref{semipr}.
 Hence, it is interesting to find out whether there is  a larger  Gevrey index for well-posedness  by exploring  the stabilizing effect of the hyperbolic perturbation $\eta\partial_t^2$ therein.

   Similar problems occur when investigating the following hyperbolic hydrostatic Navier-Stokes equations:
   \begin{equation}\label{NShy}
\left\{
\begin{aligned}
& \eta\partial_t^2 \tilde u+\partial_t\tilde u+(\tilde u\cdot\partial_x) \tilde u+ \tilde v\partial_y \tilde u\\
&\qquad\qquad +\eta\partial_t\big(  (\tilde u\cdot \partial_x )\tilde u+\tilde v\partial_y \tilde u\big) - \partial_y^2\tilde u+ \partial_x \tilde p=0,  \quad (x,y)\in \mathbb R\ \times\ ]0,1[,\\
&\partial_y\tilde p=0,  \quad (x,y)\in \mathbb R\ \times\ ]0,1[,\\
&\partial_x \tilde u +\partial_y\tilde v =0,  \quad (x,y)\in \mathbb R\ \times\ ]0,1[, \\
&\tilde u|_{y=0,1}=\tilde v|_{y=0,1}=0,  \quad  x\in \mathbb R,\\
&\tilde u|_{t=0} =\tilde u_0,\quad \partial_t\tilde u|_{t=0}=\tilde u_1,  \quad (x,y)\in \mathbb R\ \times\ ]0,1[.
\end{aligned}\right.
\end{equation}
The hydrostatic Navier-Stokes system \eqref{NShy} has a similar degeneracy feature as the Prandtl equation.  This system  that can be used to describe the large scale motion of geophysical flow plays an   important role in the atmospheric and oceanic sciences.  It is  a  limit  of the  hyperbolic  Navier-Stokes equations  \eqref{hyns}  in a thin domain  where  the vertical scale is significantly smaller than the horizontal one.  Compared to the classical Prandtl equation, much less is known for  the hydrostatic Navier-Stokes equations \eqref{NShy}. In fact, the  well-posedness of  the hydrostatic Navier-Stokes equations in Sobolev space is still unclear. Under the convex assumption,    the  Gevrey  well-posedness has been established by
G\'{e}rard-Varet-Masmoudi-Vicol \cite{MR4149066}, and later improved   by   Wang-Wang \cite{WW2022} and G\'{e}rard-Varet-Iyer-Maekawa  \cite{GIM} with
  the optimal Gevrey index  $\frac 32$. 
The aforementioned works  mainly focus on the hydrostatic equations of the  parabolic type. To the best of our knowledge, there is  no mathematical theory   on the well-posedness  of the hyperbolic hydrostatic Navier-Stokes equations. Here, we just mention some recent works attempting to explore the hyperbolic feature of  some simplified semi-linear models ( cf. \cite{2021arXiv211113052A,CMAA-1-471, MR4429384}) in order to investigage the wave type property that lead to some kind of stability effect compared to the parabolic counterparts. In addition, a  quasi-linear model was   recently studied by \cite{LPZ}.         However, the well-posedness property for the  full quasi-linear system \eqref{NShy}   remains as a challenging problem.

   This paper aims to investigate the hyperbolic Prandtl system \eqref{hyperPr} in  an anisotropic Gevrey space (see Definition \ref{degev} below). To simplify the argument, we  assume  without loss of  generality  that $\eta=1$ and  $\partial_xp=U\equiv 0$. Then we  consider 
  \begin{equation}\label{hypr}
\left\{
\begin{aligned}
&  \partial_t^2u+\partial_tu+(u\cdot\partial_x) u+v\partial_yu+ \partial_t\big(  (u\cdot \partial_x )u+v\partial_y u\big) - \partial_y^2u=0, \ \  (x,y)\in\mathbb R_+^d, \\
&\partial_x u +\partial_yv =0,\\
&u|_{y=0}=v|_{y=0}=0,\quad u|_{y\rightarrow+\infty}=0,\\
&u|_{t=0} =u_0,\quad \partial_t u|_{t=0}=u_1.
\end{aligned}\right.
\end{equation}

 \bigskip
 \noindent {\bf Notations.}  In the half-space $  \mathbb R_+^d$ with $d=2$ or $3$,   we will use  $\norm{\cdot}_{L^2}$ and $\inner{\cdot, \cdot}_{L^2}$ to denote the norm and inner product of  $L^2=L^2(\mathbb R_+^d)$   and use the notation   $\norm{\cdot}_{L_x^2}$ and $\inner{\cdot, \cdot}_{L_x^2}$  when the variable $x$ is specified. Similar notations  will be used for $L^\infty$.  And  $L^p_x L^q_y  = L^p (\mathbb R^{d-1}; L^q(\mathbb R_+))$.

 \begin{definition}\label{degev}
 The anisotropic Gevrey space	$G^{3/2,1}_{\rho,\ell}$ consists of all smooth functions $h(x,y) $ that are analytic in $y$ and of  Gevrey class $3/2$ in $x$ satisfying 
 \begin{equation*}
 	\norm{h}_{G^{3/2,1}_{\rho,\ell}}<+\infty,
 \end{equation*}
 with
 \begin{multline*}
 	\norm{h}_{G^{3/2,1}_{\rho,\ell}} ^2 \eqdef  \sum_{ m= 0}^{+\infty} \inner{  N_{\rho, m}     \norm{  \comi y^{\ell-1} \partial_x^mh}_{L^2}}^2\\
 	  +\sum_{ m= 0}^{+\infty}\sum_{k= 0}^{+\infty} \inner{(m+1) H_{\rho, m+1,k}   \norm{  \comi y^\ell  \partial_x^m\partial_y^{k} \partial_yh}_{L^2}}^2,
 \end{multline*}
 where $\comi y\eqdef (1+\abs y^2)^{1/2}$ and  the number $\ell\geq 2$ is   given   and 
 \begin{equation}\label{nl}
 	 H_{\rho,m,k}= \frac{\rho^{m+k+1}(m+k+1)^9}{(m+k)! (m!)^{1/2}}  \ \textrm{ and }\  N_{\rho,m}=H_{\rho,m,0}=\frac{\rho^{m+1}(m+1)^9}{(m!)^{3/2}}.
 \end{equation}
 \end{definition}

The main result of the paper is stated as follows.

 \begin{theorem}\label{thm:main}
 If  the initial data of the hyperbolic Prandtl system \eqref{hypr}   satisfy   $u_0\in G^{3/2,1}_{2\rho_0,\ell}$ and  $u_1\in G^{3/2,1}_{2\rho_0,\ell+1}$ for some $\rho_0>0$ and are compatible to the boundary conditions in \eqref{hypr}.  Then  problem \eqref{hypr} admits a unique local solution $u\in L^\infty\big([0,T]; \ G^{3/2,1}_{\rho,\ell}\big)$ for some $T>0$ and $0<\rho\leq \rho_0$.
 \end{theorem}

 The key part in the proof of Theorem \ref{thm:main}  is to derive  the {\it  a  priori} estimate for   \eqref{hypr}  so that  the existence and uniqueness   follow  from  a standard argument.  Hence,  for brevity, we only present the proof of  the {\it a priori} estimate.
 
  The paper is organized as follows.  Sections \ref{sec:aux}-\ref{sec:comple} are for proving the
   {\it a priori} estimate in  2D. The proof in 3D  will then be presented in Section \ref{sec:3D}.     

\section{The \emph{a priori} estimate in 2D}\label{sec:aux}

In this section we will state the \emph{a priori} estimate for the  hyperbolic Prandtl system \eqref{hypr} when $d=2$ and its proof  will be given in  Sections \ref{sec:tange}-\ref{sec:comple}. 

  In the following argument we assume  the initial data    in \eqref{hypr} satisfy that  $u_0\in G^{3/2,1}_{ 2\rho_0,\ell}$ and $u_1\in G^{3/2,1}_{ 2\rho_0,\ell+1}$  for some $\rho_0> 0$, and suppose  that $u\in L^\infty  \big([0, T];\  G^{3/2,1}_{ \rho,\ell}\big)$  solves \eqref{hypr}
  where 
 	 \begin{equation}\label{rho}
	\rho=\rho(t)\eqdef \rho_0e^{-\mu t},\quad 0\leq t\leq T,
	\end{equation}
with $\mu>1$ beging  a given  large constant to be determined later.   By using the notation  that
\begin{equation}\label{varp}
\varphi=	\partial_tu+u\partial_x u+v\partial_yu  \  \textrm{  with  }  \  v(t, x,y)=-\int_0^y \partial_x u(t,x,\tilde y) d\tilde y,
\end{equation}
we can reformulate system \eqref{hypr} in  2D 
as
\begin{equation}\label{neweq}
	\left\{
\begin{aligned}
&\partial_tu+u\partial_x u+v\partial_yu=\varphi,\\
&  \partial_t\varphi+\varphi- \partial_y^2u=0,    \\
&u|_{y=0}=\varphi|_{y=0}=0,\\
&u|_{t=0} =u_0,\quad \varphi|_{t=0}=\varphi_0,
\end{aligned}\right.
\end{equation}
where
\begin{equation}\label{vazero}
	\varphi_0=u_1+u_0\partial_xu_0-(\partial_yu_0)\int_0^y \partial_xu_0(x,\tilde y)d\tilde y.
\end{equation}
As for the classical  Prandtl equation, the loss of one order tangential derivatives occurs in the first equation of  \eqref{neweq}.  To overcome this difficulty,  inspired by \cite{MR4465902} we introduce two auxiliary functions $\mathcal U$ and $\lambda$. Precisely, let
 $\mathcal U$  be a solution to the Cauchy problem
\begin{eqnarray}\label{mau}
\left\{
\begin{aligned}
& \inner{\partial_t+u\partial_x+v\partial_y}   \int_0^y\mathcal U  d\tilde y   =  -\partial_x v,\\
& \mathcal U|_{t=0}=0.
\end{aligned}
\right.
\end{eqnarray}
The existence of $\mathcal U$  follows from the standard  theory of transport equations. In fact, one can
 first apply the existence theory for  linear transport equations to  construct a solution $f$ to the  Cauchy problem
\begin{eqnarray*}
	\left\{
\begin{aligned}
&\inner{\partial_t+u\partial_x+v\partial_y}f=  -\partial_x v,\\
&  f|_{y=0}=0,\\
&f|_{t=0}=0,
\end{aligned}
\right.
\end{eqnarray*}
and then set $f=\int_0^y \mathcal U(t,x,\tilde y) d\tilde y$.

 By virtue of  $\mathcal U$ and
 \begin{eqnarray}\label{ma}
  \lambda\eqdef \partial_xu-   (\partial_yu)\int_0^y\mathcal U(t,x,\tilde y)  d\tilde y ,
\end{eqnarray}
we can cancel the  term involving  $v$ with  the  highest tangential derivative
as shown in the following equation \eqref{lateq}.  The two auxiliary functions
have the   relation
 \begin{eqnarray}\label{ymau}
\begin{aligned}
\inner{\partial_t+u\partial_x+v\partial_y} \mathcal U
= \partial_x\lambda + (\partial_x\partial_yu)\int_0^y\mathcal U(t,x,\tilde y)  d\tilde y+(\partial_xu)\mathcal U.
\end{aligned}
\end{eqnarray}
In addition,  we apply  $\partial_x$ to the first equation in \eqref{neweq} and multiply \eqref{mau} by $\partial_y u$. Then the subtraction of these two equations  yields the following equation for $\lambda:$
\begin{equation}\label{lateq}
\big(\partial_t  +u\partial_x +v\partial_y  \big)\lambda
	= \partial_x\varphi  -(\partial_xu)\partial_xu-  (\partial_y\varphi)  \int_0^y\mathcal U d\tilde y.
\end{equation}

   \begin{definition}\label{defabnorm}
Let $\ell\geq2$ be the number given in Definition \ref{degev}, and let $\varphi, \mathcal U, \lambda$ be given  in \eqref{varp}, \eqref{mau} and \eqref{ma}, respectively. 	By denoting
	\begin{equation*}
		\vec a=(u,  \mathcal U,   \lambda, \varphi),
	\end{equation*}
	we define $|\vec a|_{X_\rho} $ and $ |\vec a|_{Y_\rho} $ by
	 \begin{equation}\label{defX}
\begin{aligned}
		& |\vec a|_{X_{\rho}}^2 =   \sum_{ m=0}^{\infty}   N_{\rho,m+1}^2  \norm{ \partial_x^m\mathcal U}_{L^2}^2  + \sum_{ m=0}^{\infty}    (m+1)     N_{\rho,m+1}^2    \norm{\comi y^{\ell-1}\partial_x^m \lambda}_{L^2}^2  \\
	 & \qquad\quad +  \sum_{m,  k\geq 0}       (m+1)^2 H_{\rho,m+1,k}^2\inner{ \norm{  \comi y^\ell  \partial_x^m\partial_y^k\varphi }_{L^2}^2+\norm{  \comi y^\ell  \partial_x^m\partial_y^{k}\partial_y u}_{L^2}^2},
	\end{aligned}
	\end{equation}
	and
	\begin{equation}\label{defY}
	\begin{aligned}
		&|\vec a|_{Y_{\rho}}^2 =  \sum_{ m=0}^{\infty}   (m+1)N_{\rho,m+1}^2  \norm{ \partial_x^m\mathcal U}_{L^2}^2  +\sum_{ m=0}^{\infty}   (m+1)^2N_{\rho,m+1}^2          \norm{  \comi y^{\ell-1}\partial_x^m \lambda}_{L^2}^2  \\
		 & + \sum_{m,  k\geq 0}      \inner{ m+k+1}   (m+1)^2 H_{\rho,m+1,k} ^2\Big(\norm{  \comi y^\ell  \partial_x^m\partial_y^k\varphi }_{L^2}^2+\norm{  \comi y^\ell  \partial_x^m\partial_y^{k} \partial_y u}_{L^2}^2\Big).
		\end{aligned}
	\end{equation}
	Recall  that  $N_{\rho,m}$ and $H_{\rho,m,k}$ are given in \eqref{nl} with $\rho$     defined in  \eqref{rho}.
\end{definition}

\begin{remark} From \eqref{defX} and \eqref{defY}, it follows directly that
\begin{equation}\label{XY}
	|\vec a|_{X_{\rho}}\leq |\vec a|_{Y_{\rho}}.
\end{equation}
Moreover, as shown in Lemma \ref{lem:tang} below,
\begin{align*}
\norm{u}_{G^{3/2,1}_{\rho,\ell}}\leq C		|\vec a|_{X_{\rho}},
\end{align*}
where $C$ is a constant depending only on $\rho_0,\ell$ and the Sobolev embedding  constants.
Denote $\vec a(0)=\vec a|_{t=0}.$ Then there exists a constant $C_0>0$ such that
	\begin{equation}\label{init}
		\abs{\vec a(0)}_{X_{\rho_0}}\leq  C_0 \Big(\norm{u_0}_{G^{3/2,1}_{2\rho_0,\ell}}+\norm{u_1}_{G^{3/2,1}_{2\rho_0,\ell+1}}\Big).
	\end{equation}
	For clear presentation, the proof of \eqref{init} will be given in Appendix \ref{sec:app}.
\end{remark}

We can now state the theorem on the a priori estimate. 

\begin{theorem}[\emph{A priori} estimate]
 	\label{thm:apriori}
 	Under the same assumption as  in Theorem \ref{thm:main},   there exists  a 
 	constant $\mu \geq 1$ depending only on $\ell,\rho_0$,   the Sobolev embedding constants and   
 	the intial data
 	such that if
 	\begin{equation}\label{bootass}
 		\sup_{0\leq t\leq T}\abs{\vec a(t)}_{X_\rho}+ \Big(\int_0^T \abs{\vec a(t)}_{Y_\rho}^2 dt\Big)^{1/2} \leq 2 C_0 \Big(\norm{u_0}_{G^{3/2,1}_{2\rho_0,\ell}}+\norm{u_1}_{G^{3/2,1}_{2\rho_0,\ell+1}}\Big)
 	\end{equation}
 	with $\rho$ defined by \eqref{rho} and $T=\mu^{-1}$,  then
 	\begin{equation}\label{dssert}
 		\sup_{0\leq t\leq T}\abs{\vec a(t)}_{X_\rho}+ \Big(\int_0^T \abs{\vec a(t)}_{Y_\rho}^2 dt\Big)^{1/2}\leq C_0 \Big(\norm{u_0}_{G^{3/2,1}_{2\rho_0,\ell}}+\norm{u_1}_{G^{3/2,1}_{2\rho_0,\ell+1}}\Big).
 	\end{equation}
 \end{theorem}

The proof of Theorem \ref{thm:apriori} will be given in  Sections \ref{sec:tange}-\ref{sec:comple}. We first list some facts for later use.
In view of \eqref{nl} and \eqref{rho},  we have
\begin{equation}\label{dev:rho}
       \forall\ m,k\geq 0,\ 		\frac{d}{dt}N_{\rho,m}=-\mu (m+1) N_{\rho,m}\ \textrm{ and }\ \frac{d}{dt}H_{\rho,m,k}=-\mu (m+k+1) H_{\rho,m,k},
 \end{equation}
 and
 \begin{equation}\label{loup}
 	\forall \ 0\leq t\leq T=\mu^{-1},\quad  e^{-1}\rho_0\leq  \rho(t)\leq \rho_0.
 \end{equation}
 Similarly, if we denote
 \begin{equation}
 \label{lrho}
 L_{\rho,k}\eqdef 	H_{\rho,1,k}=\frac{\rho^{k+2}(k+2)^9}{(k+1)!},
 \end{equation}
then
\begin{equation}
	\label{devlprho}
	 \forall\ k\geq 0,\quad 		\frac{d}{dt}L_{\rho,k}=-\mu (k+2) L_{\rho,k}.
\end{equation}
 We will use   the following    Young's inequality for   discrete convolution:
 \begin{eqnarray}\label{dis}
\bigg[ \sum_{m=0}^{\infty} \Big(\sum_{j=0}^m p_jq_{m-j}\Big)^2\bigg]^{1/2}\leq \Big(\sum_{m=0}^{\infty}   q_m^2 \Big)^{1/2} \sum_{j=0}^{\infty}   p_j,
\end{eqnarray}
where    $ \set{p_{j}}_{j\geq0} $ and $\set {q_{j}}_{j\geq0} $ are positive sequences.

 \section{Tangential derivatives of $u$} \label{sec:tange}

  To simplify the notations, we will use  $C$ in Sections \ref{sec:tange}-\ref{sec:comple}    to denote a generic constant  which may vary from line to line and  depend only on  $ \ell $, $\rho_0$ and the Sobolev embedding constants, but is  independent of $\mu$ in \eqref{rho} and  the order of derivatives.

Compared with Definition \ref{degev}, the  tangential derivatives of $u$ are not specified in the    definitions \eqref{defX} and \eqref{defY} of $|\vec a|_{X_\rho}$ and $|\vec a|_{Y_\rho}$. As a preliminary step to prove Theorem \ref{thm:apriori}, we first use \eqref{ma} to control the tangential derivatives of $u$ in terms of $|\vec a|_{X_\rho} $ and $|\vec a|_{Y_\rho} $.

\begin{lemma}\label{lem:tang}
 Under the assumptions  of  Theorem \ref{thm:apriori}, we have
	\begin{align*}
		  \sum_{ m=0}^{\infty}   N_{\rho, m}  ^2   \norm{  \comi y^{\ell-1} \partial_x^m u}_{L^2}^2 \leq    C\big(1+ |\vec a|_{X_{\rho}}^2\big) |\vec a|_{X_{\rho}}^2,
	\end{align*}
	and
	\begin{align*}
		  \sum_{ m=0}^{\infty}  (m+1) N_{\rho, m}  ^2   \norm{  \comi y^{\ell-1} \partial_x^m u}_{L^2}^2 \leq    C  \big(1+ |\vec a|_{X_{\rho}}^2\big) |\vec a|_{Y_{\rho}}^2.
	\end{align*}	
	Recall that $N_{\rho,m}$ is defined by \eqref{nl}.
\end{lemma}

\begin{proof}We first prove the first statement.
	In view of  \eqref{ma},  we have
	\begin{align*}
		\partial_x^{m+1} u=\partial_x^m\lambda+ \sum_{j=0}^m {m\choose j}(\partial_x^j\partial_yu)\int_0^y\partial_x^{m-j}\mathcal U(t,x,\tilde y)  d\tilde y.	
\end{align*}
Thus, we first multiply both sides of the above equation by $\comi y^{\ell-1}$ and then take the $L^2$-product with $\comi y^{\ell-1}\partial_x^{m+1}u$. Then this with the fact that
\begin{align*}
\forall\ (x,y)\in   \mathbb R_+^2,\quad 	\Big| \int_0^y\partial_x^{m-j}\mathcal U(t,x,\tilde y)  d\tilde y \Big| \leq  \comi y^{1/2} \norm{\partial_x^{m-j}\mathcal U(x,\cdot)}_{L_y^2}
\end{align*}
 implies
\begin{equation}\label{i1i2}
\sum_{ m=0}^{\infty}   N_{\rho, m+1}  ^2  \norm{  \comi y^{\ell-1} \partial_x^{m+1}u}_{L^2}^2\leq    I_1+I_2,
\end{equation}
where 
\begin{equation}\label{i1}
	I_1=\sum_{ m=0}^{+\infty}   N_{\rho, m+1}  ^2    \norm{  \comi y^{\ell-1} \partial_x^{m}\lambda}_{L^2} ^2 \leq |\vec a|_{X_\rho}^2
\end{equation}
due to \eqref{defX}, and
\begin{multline}\label{i2}
	I_2= 2\sum_{ m=0}^{+\infty} \bigg[\sum_{j=0}^{[m/2]}  {m\choose j} N_{\rho, m+1}    \norm{\comi y^{\ell} \partial_x^j\partial_yu}_{L_x^\infty L_y^2} \norm{\partial_x^{m-j}\mathcal U}_{L^2} \bigg]^2\\
	+2\sum_{m=0}^{+\infty} \bigg[\sum_{j= [m/2]+1}^m  {m\choose j} N_{\rho, m+1}    \norm{\comi y^{\ell} \partial_x^j\partial_yu}_{L^2} \norm{\partial_x^{m-j}\mathcal U}_{L_x^\infty L_y^2} \bigg]^2.
\end{multline}
As usual  $[m/2]$ represents the largest integer less than or equal to $ m/2$.
To estimate $I_2$, we first write
\begin{align*}
	&\sum_{j=0}^{[m/2]}  {m\choose j} N_{\rho, m+1}    \norm{\comi y^{\ell} \partial_x^j\partial_yu}_{L_x^\infty L_y^2} \norm{\partial_x^{m-j}\mathcal U}_{L^2} \\
	&\qquad=\sum_{j=0}^{[m/2]} \frac{m!}{j!(m-j)!} \frac{N_{\rho, m+1}}{N_{\rho, j+3} N_{\rho,m-j+1} }  \\
	&\qquad\qquad\qquad\qquad\times \inner{ N_{\rho, j+3} \norm{\comi y^{\ell} \partial_x^j\partial_yu}_{L_x^\infty L_y^2} } \inner{ N_{\rho,m-j+1} \norm{\partial_x^{m-j}\mathcal U}_{L^2}}.
\end{align*}
By using  the fact that (see Appendix \ref{sec:app} for its proof)
\begin{equation}\label{fe1}
	\forall\  j\leq [m/2],\quad  \frac{m!}{j!(m-j)!} \frac{N_{\rho, m+1}}{N_{\rho, j+3} N_{\rho,m-j+1} }  \leq \frac{C}{j+1}
\end{equation}
and the  Young's inequality \eqref{dis}   for   discrete convolution,  we obtain
\begin{equation}\label{YE}
\begin{aligned}
	&\sum_{j=0}^{m}  \bigg[\sum_{j=0}^{[m/2]}  {m\choose j} N_{\rho, m+1}    \norm{\comi y^{\ell} \partial_x^j\partial_yu}_{L_x^\infty L_y^2} \norm{\partial_x^{m-j}\mathcal U}_{L^2}\bigg]^2\\
	&\leq C \sum_{j=0}^{m}  \bigg[ \sum_{j=0}^{[m/2]}  \frac{ N_{\rho, j+3} \norm{\comi y^{\ell} \partial_x^j\partial_yu}_{L_x^\infty L_y^2} } {j+1}\inner{ N_{\rho,m-j+1} \norm{\partial_x^{m-j}\mathcal U}_{L^2}}\bigg]^2\\
	& \leq C \Big(\sum_{j=0}^{+\infty} \frac{ N_{\rho, j+3} \norm{\comi y^{\ell} \partial_x^j\partial_yu}_{L_x^\infty L_y^2} }{j+1}\Big)^2\sum_{j=0}^{+\infty}   N_{\rho,j+1}^2 \norm{\partial_x^{j}\mathcal U}_{L^2}^2\\
	& \leq C \Big(\sum_{j=1}^{+\infty} j^{-2}  \Big)  \Big(\sum_{j=0}^{+\infty}  N_{\rho, j+3} ^2\norm{\comi y^{\ell} \partial_x^j\partial_yu}_{L_x^\infty L_y^2} ^2\Big)|\vec a|_{X_\rho}^2\leq C|\vec a|_{X_\rho}^4,
\end{aligned}
\end{equation}
where in the last inequality we have  used the Sobolev embedding and  \eqref{defX}. 
 
 Similarly,  by using the fact that (see Appendix \ref{sec:app} for its proof)
\begin{equation}\label{fe2}
	\forall\   [m/2]+1\leq j\leq m, \quad  \frac{m!}{j!(m-j)!} \frac{N_{\rho, m+1}}{N_{\rho,j+1} N_{\rho,m-j+3} }  \leq \frac{C}{m-j+1},
\end{equation}
we have
\begin{equation*}
\begin{aligned}
	&\sum_{m=0}^{+\infty} \bigg[\sum_{j= [m/2]+1}^m  {m\choose j} N_{\rho, m+1}    \norm{\comi y^{\ell} \partial_x^j\partial_yu}_{L^2} \norm{\partial_x^{m-j}\mathcal U}_{L_x^\infty L_y^2} \bigg]^2\\
	&\leq \sum_{m=0}^{+\infty} \bigg[\sum_{j> [m/2]}^m    \big (N_{\rho, j+1}    \norm{\comi y^{\ell} \partial_x^j\partial_yu}_{L^2}\big)  \frac{N_{\rho,m-j+3}\norm{\partial_x^{m-j}\mathcal U}_{L_x^\infty L_y^2}}{m-j+1} \bigg]^2  \leq  C|\vec a|_{X_\rho}^4.
	\end{aligned}
\end{equation*}
This together with \eqref{YE} and \eqref{i2}  yields
\begin{equation}\label{pri2}
	I_2\leq C|\vec a|_{X_\rho}^4.
\end{equation}
Thus, by substituting the above estimate and \eqref{i1} into \eqref{i1i2}, we get
\begin{align*}
		\sum_{m=0}^{+\infty} N_{\rho, m}  ^2  \norm{  \comi y^{\ell-1} \partial_x^{m}u}_{L^2}^2&= N_{\rho, 0}  ^2  \norm{  \comi y^{\ell-1} u}_{L^2}^2+	\sum_{m=0}^{+\infty} N_{\rho, m+1}  ^2  \norm{  \comi y^{\ell-1} \partial_x^{m+1}u}_{L^2}^2\\
		&\leq N_{\rho, 0}  ^2  \norm{  \comi y^{\ell-1} u}_{L^2}^2+C \inner{ |\vec a|_{X_{\rho}}^2+|\vec a|_{X_{\rho}}^4}.
\end{align*}
 On the other hand, by Hardy's inequality and \eqref{loup}, we have 
 \begin{align*}
 	N_{\rho,0}\norm{\comi y^{\ell-1}u}_{L^2}\leq C \frac{N_{\rho,0}}{N_{\rho, 1}}N_{\rho, 1}\norm{\comi y^{\ell}\partial_y u}_{L^2}\leq   \frac{C}{\rho}N_{\rho, 1}\norm{\comi y^{\ell}\partial_y u}_{L^2} \leq C |\vec a|_{X_\rho}.
 \end{align*}
Thus, the first statement in the lemma  follows.

The second statement can be proved similarly. In fact,  firstly we have
\begin{align*}
		 \sum_{m=0}^{+\infty} (m+1)N_{\rho, m}  ^2  \norm{  \comi y^{\ell-1} \partial_x^{m}u}_{L^2}^2 \leq C |\vec a|_{X_\rho}^2+	\sum_{m=1}^{+\infty} (m+1)N_{\rho, m}  ^2  \norm{  \comi y^{\ell-1} \partial_x^{m}u}_{L^2}^2.
\end{align*}
Similar to the estimates \eqref{i1i2} and \eqref{i1}, by using  the fact that $ m+1 \leq C (m-j)$ for $j\leq [m/2]$ and  $ m+1 \leq C j$ for $  [m/2]+1\leq j\leq m,$ we have by  following  the argument for \eqref{pri2} that
\begin{align*}
		& \sum_{m=1}^{+\infty} (m+1)N_{\rho, m}  ^2  \norm{  \comi y^{\ell-1} \partial_x^{m}u}_{L^2}^2\leq C \sum_{m=0}^{+\infty} (m+1)N_{\rho, m+1}  ^2  \norm{  \comi y^{\ell-1} \partial_x^{m+1}u}_{L^2}^2\\
		&\leq  C\sum_{ m=0}^{+\infty} \bigg[\sum_{j=0}^{[m/2]}  {m\choose j} N_{\rho, m+1}    \norm{\comi y^{\ell} \partial_x^j\partial_yu}_{L_x^\infty L_y^2} (m-j+1)^{1\over2}\norm{\partial_x^{m-j}\mathcal U}_{L^2} \bigg]^2\\
	&\quad+C\sum_{m=0}^{+\infty} \bigg[\sum_{j= [m/2]+1}^m  {m\choose j} N_{\rho, m+1} (j+1)^{1/2}   \norm{\comi y^{\ell} \partial_x^j\partial_yu}_{L^2} \norm{\partial_x^{m-j}\mathcal U}_{L_x^\infty L_y^2} \bigg]^2\\
	&\leq  C |\vec a|_{X_\rho}^2  |\vec a|_{Y_\rho}^2.\end{align*}
Combining the above estimates gives
\begin{align*}
	\sum_{m=0}^{+\infty} (m+1)N_{\rho, m}  ^2  \norm{  \comi y^{\ell-1} \partial_x^{m}u}_{L^2}^2\leq C \inner{1+|\vec a|_{X_\rho}^2} |\vec a|_{Y_\rho}^2.
\end{align*}
 Then the proof of the lemma  is completed.
\end{proof}

As a direct consequence of the above lemma, we have the following corollary by using
\begin{align*}
	\norm{  \partial_x^m  v}_{L_x^2L_y^\infty} \leq  C \norm{ \comi y^{\ell-1} \partial_x^{m+1} u}_{L^2}.
\end{align*}

\begin{corollary}\label{cor:v}
Under the assumptions of  Theorem  \ref{thm:apriori}, we have
\begin{align*}
		  \sum_{ m=0}^{\infty}   N_{\rho, m+1}  ^2   \norm{  \partial_x^m  v}_{L_x^2L_y^\infty}^2 \leq    C\big(1+ |\vec a|_{X_{\rho}}^2\big) |\vec a|_{X_{\rho}}^2,
	\end{align*}
	and
	\begin{align*}
		  \sum_{ m=0}^{\infty}  (m+1) N_{\rho, m+1}  ^2   \norm{   \partial_x^m v}_{L_x^2L_y^\infty}^2 \leq    C\big(1+ |\vec a|_{X_{\rho}}^2\big) |\vec a|_{Y_{\rho}}^2.
	\end{align*}	
\end{corollary}

   \section{Gevrey norm of $\varphi$ and $\partial_yu$}\label{sec:var}

   In this section,  we   will study  the last term in definition \eqref{defX} of $|\vec a|_{X_\rho}$  involving  the mixed derivatives of $\varphi$ and $\partial_yu$.  The estimate is stated in the following proposition. 

   \begin{proposition}\label{prp:varu}
   Under the assumptions of  Theorem \ref{thm:apriori}, we have
   	\begin{align*}
	&\frac12\frac{d}{dt} \sum_{m=0}^{+\infty}\sum_{k=0}^{+\infty} (m+1)^2 H_{\rho,m+1,k}^2  \Big( \norm{ \comi y^\ell  \partial_x^m \partial_y^k  \varphi }_{L^2}^2+   \norm{\comi y^\ell  \partial_x^m \partial_y^{k} \partial_y u }_{L^2}^2\Big) \\
	 &\leq C  \big(1+|\vec a|_{X_\rho}^4 \big) |\vec a|_{Y_\rho}^2-\mu \sum_{m,k\geq 0}  \inner{m+k+1}(m+1)^2 H_{\rho,m+1,k}^2\\
	 & \qquad\qquad\qquad\qquad\qquad\quad \qquad\qquad \times \Big ( \norm{ \comi y^\ell  \partial_x^m \partial_y^k \varphi }_{L^2}^2+   \norm{\comi y^\ell  \partial_x^m\partial_y^{k} \partial_y u }_{L^2}^2\Big  ).\end{align*}
	     \end{proposition}

To have a clear presentation,  we first deal with the  normal and tangential derivatives in Subsections \ref{subsec:nor} and \ref{subsec:tan} respectively. The estimate on the  mixed derivatives will then be  presented in the last subsection.

   \subsection{Normal derivatives}\label{subsec:nor}  We first prove the following estimate  on the normal derivatives of $\varphi$ and $\partial_yu$.

   \begin{lemma}\label{lem:nor}
   Under the assumptions of  Theorem \ref{thm:apriori}, we have
   	\begin{align*}
&	\frac{1}{2}\frac{d}{dt} \sum_{k=0}^{+\infty}L_{\rho , k}^2 \inner{\norm{\comi y^\ell  \partial_y^{k}\varphi}_{L^2}^2+\norm{\comi y^\ell  \partial_y^{k}\partial_y u}_{L^2}^2} \\
	  &\leq C\big( 1+  |\vec a|_{X_{\rho}}^2\big)|\vec a|_{Y_{\rho}}^2-\mu \sum_{k=0}^{+\infty} \inner{k+1} L_{\rho , k}^2 \inner{\norm{\comi y^\ell  \partial_y^{k}\varphi}_{L^2}^2+\norm{\comi y^\ell  \partial_y^{k}\partial_y u }_{L^2}^2}
	  ,
\end{align*}
where $L_{\rho,k}$ is defined by \eqref{lrho}, that is,
\begin{equation*}
	L_{\rho,k}\eqdef H_{\rho,1,k}=\frac{\rho^{k+2}(k+2)^9}{(k+1)!}.
\end{equation*}
   \end{lemma}

   \begin{proof}
 We   apply   $\comi y^\ell \partial_y^{k+1}$ and $\comi y^\ell \partial_y^{k}$  to the first and the second equations in \eqref{neweq} respectively to have 
 that
 \begin{equation}\label{varyu}
 \left\{
 \begin{aligned}
 &\big(\partial_t+u\partial_x+v\partial_y\big)   \comi y^\ell   \partial_y^{k+1}   u=  \comi y^\ell  \partial_y^{k+1}   \varphi  +\mathcal R_{k},\\
 &	(\partial_t+1)\comi y^\ell \partial_y^{k}\varphi=\comi y^\ell\partial_y^{k+2}u,
 \end{aligned}
 \right.
 \end{equation}
 where
\begin{equation}\label{errrk}
	 \mathcal R_k =  v\big(\partial_y  \comi y^\ell  \big)  \partial_y^{k+1}   u -\sum_{ i=1}^{k+1} {{k+1}\choose i}\comi y^\ell \Big[(  \partial_y^i u) \partial_x\partial_y^{k+1-i}u +   ( \partial_y^i v) \partial_y^{k+2-i} u\Big].
\end{equation}
Then we take  the $L^2$-product with  $\comi y^\ell \partial_y^{k+1}u$  for the first equation in \eqref{varyu}, and with $ \comi y^\ell \partial_y^{k}\varphi$ for the second one.
By using  
\begin{multline*}
	\inner{\comi y^\ell  \partial_y^{k+2}   u,\  \comi y^\ell  \partial_y^{k}   \varphi}_{L^2}=-\inner{\comi y^\ell  \partial_y^{k+1}   u,\  \comi y^\ell  \partial_y^{k+1}   \varphi}_{L^2} \\
	- \inner{( \partial_y\comi y^{2\ell})\partial_y^{k+1}   u,\    \partial_y^{k}   \varphi}_{L^2}
	 -  \int_{\mathbb R} \big[(\partial_y^{k+1}   u) \partial_y^{k}   \varphi\big]\big |_{y=0} dx,
\end{multline*}
we obtain
\begin{equation}\label{neqva}
\begin{aligned}
&	 \frac{1}{2}\frac{d}{dt} \inner{ \norm{\comi y^\ell  \partial_y^{k}\varphi}_{L^2}^2+\norm{\comi y^\ell  \partial_y^{k+1}u}_{L^2}^2}+ \norm{\comi y^\ell  \partial_y^{k}\varphi}_{L^2}^2\\
	  	&  =  \inner{\mathcal R_k,   \comi y^\ell \partial_y^{k+1}   u}_{L^2} - \inner{( \partial_y\comi y^{2\ell})\partial_y^{k+1}   u,\    \partial_y^{k}   \varphi}_{L^2} - \int_{\mathbb R} \big[(\partial_y^{k+1}   u) \partial_y^{k}   \varphi\big]\big |_{y=0} dx.
\end{aligned}
\end{equation}
In view of   definitions  \eqref{defX}-\eqref{defY} of $|\vec a|_{X_\rho}$ and $|\vec a|_{Y_\rho}$ and \eqref{XY},  we have 
\begin{align*}
	-\sum_{k=0}^{+\infty}L_{\rho , k}^2\inner{( \partial_y\comi y^{2\ell})\partial_y^{k+1}   u,\    \partial_y^{k}   \varphi}_{L^2}\leq C |\vec a|_{X_{\rho }}^2\leq C |\vec a|_{Y_{\rho }}^2.
\end{align*}
This   together with the  Sobolev inequality  $\norm{f}_{L_y^\infty}^2\leq 2\norm{f}_{L_y^2}\norm{\partial_y f}_{L_y^2}$ implies
\begin{align*}
	&-\sum_{k=0}^{+\infty}  L_{\rho , k}^2\int_{\mathbb R} \big[(\partial_y^{k+1}   u) \partial_y^{k}   \varphi\big]\big |_{y=0} dx  \\
	&\leq  C \sum_{k=0}^{+\infty}  \inner{k+1} L_{\rho , k} L_{\rho ,k+1} \inner{\norm{ \partial_y^{k}\varphi}_{L^2} \norm{ \partial_y^{k+1}  \varphi}_{L^2}+\norm{ \partial_y^{k+1}   u}_{L^2} \norm{ \partial_y^{k+2}   u}_{L^2}  }\\
	&\leq C \sum_{k=0}^{+\infty}\inner{k+1}   L_{\rho , k}^2\inner{ \norm{ \partial_y^{k}\varphi}_{L^2}^2+ \norm{ \partial_y^{k }   \partial_y u}_{L^2}^2}\leq C|\vec a|_{Y_{\rho}}^2,
\end{align*}
where in the first inequality  we have used the fact that $L_{\rho, k}/L_{\rho,k+1}=(k+2)/\rho\leq C(k+1)$ due to \eqref{loup}.
As a result, by  the above estimates and \eqref{devlprho}, we multiply equation \eqref{neqva} by $L_{\rho,k}^2$ and then take the summation over $k\geq 0$  to obtain 
\begin{equation}\label{noebou}
\begin{aligned}
	&\frac{1}{2}\frac{d}{dt} \sum_{k=0}^{+\infty}L_{\rho , k}^2\inner{ \norm{\comi y^\ell  \partial_y^{k}\varphi}_{L^2}^2+\norm{\comi y^\ell  \partial_y^{k}\partial_yu}_{L^2}^2} \\
	& \leq -\mu \sum_{k=0}^{+\infty} \inner{k+1} L_{\rho , k}^2\inner{ \norm{\comi y^\ell  \partial_y^{k}\varphi}_{L^2}^2+\norm{\comi y^\ell  \partial_y^{k}\partial_yu}_{L^2}^2}\\
	&\quad+\sum_{k=0}^{+\infty}L_{\rho , k}^2 \inner{\mathcal R_k,\  \comi y^\ell \partial_y^{k+1}   u}_{L^2}+C |\vec a|_{Y_{\rho }}^2,
\end{aligned}
\end{equation}
where $\mathcal R_m$ is given by \eqref{errrk}. 
For the term in the last inequality, we claim that 
\begin{equation}\label{esterror}
	\sum_{k=0}^{+\infty}L_{\rho , k}^2 \inner{\mathcal R_k,\  \comi y^\ell \partial_y^{k+1}   u}_{L^2} \leq    C\big(1+|\vec a|_{X_{\rho}}^2\big) |\vec a|_{Y_{\rho}}^2.
\end{equation}
The proof of \eqref{esterror} is postponed after the proof of the lemma. Now with the claim and  the above two estimates, we  complete the proof of lemma.
\end{proof}

\begin{proof}[Proof of assertion \eqref{esterror}]
We use the estimate
\begin{align*}
	  \inner{\mathcal R_k,   \comi y^\ell \partial_y^{k+1}   u}_{L^2}
	  \leq
	 \inner{k+1}      \norm{ \comi y^\ell  \partial_y^{k+1}   u }_{L^2}^2+  \Big[ \inner{k+1}^{-1/2}     \norm{ \mathcal R_k}_{L^2}\Big]^2
\end{align*}
and definition \eqref{defY} of $|\vec a|_{Y_\rho}$ to obtain
\begin{equation}\label{asak}
	\sum_{k=0}^{+\infty}L_{\rho , k}^2 \inner{\mathcal R_k,\  \comi y^\ell \partial_y^{k+1}   u}_{L^2}
	   \leq
	 |\vec a|_{Y_\rho}^2+ \sum_{k=0}^{+\infty}\Big[ \inner{k+1}^{-1/2} L_{\rho , k}     \norm{ \mathcal R_k}_{L^2}\Big]^2.
\end{equation}
Moreover, in view of \eqref{errrk}, it follows that
\begin{equation}\label{rk}
\begin{aligned}
	&\sum_{k=0}^{+\infty}\Big[ \inner{k+1}^{-1/2} L_{\rho , k}     \norm{ \mathcal R_k}_{L^2}\Big]^2\\
	&\leq C |\vec a|_{X_{\rho }}^3+ \sum_{k=0}^{+\infty} \bigg[ (k+1)^{-1/2} L_{\rho , k}    \sum_{ i=1}^{k+1} {{k+1}\choose i}   \norm{\comi y^\ell(  \partial_y^i u)\partial_x \partial_y^{k+1-i} u}_{L^2}\bigg]^2  \\
	& \quad+\sum_{k=0}^{+\infty} \bigg[   (k+1)^{-1/2}  L_{\rho , k}   \sum_{ i=1}^{k+1} {{k+1}\choose i}  \norm{\comi y^\ell (  \partial_y^i v) \partial_y^{k+2-i} u}_{L^2}\bigg]^2.
\end{aligned}
\end{equation}
For the last term the right-hand side, we use the decomposition
\begin{align*}
	\sum_{ i=1}^{k+1}= \sum_{ i=1}^{[(k+1)/2]}+\sum_{ i=[(k+1)/2]+1}^{k+1},
\end{align*}
to write
\begin{equation}\label{ambm}
	   (k+1)^{-1/2}  L_{\rho , k}   \sum_{ i=1}^{k+1} {{k+1}\choose i}  \norm{\comi y^\ell (  \partial_y^i v) \partial_y^{k+2-i} u}_{L^2}  \leq p_k+q_k,
	  \end{equation}
with
\begin{eqnarray*}
	\begin{aligned}
		 p_k&=		   \sum_{i=1}^{[(k+1)/2]}   \frac{(k+1)!}{i!(k+1-i)!}  \frac{(k+1)^{-1/2} L_{\rho ,k} }{H_{\rho,4, i-1}L_{\rho ,k+1-i}  }  \\
		 &\qquad\qquad\qquad\quad \times \big( H_{\rho,4, i-1} \norm{ \partial_y^{i} v}_{L^\infty } \big)\big(L_{\rho , k+1-i} \norm{ \comi y^{\ell}  \partial_y^{k+2-i}u}_{L^2}\big),\\
			q_k&=    \sum_{i=[(k+1)/2]+1}^{k+1}\frac{(k+1)!}{i!(k+1-i)!} \frac{(k+1)^{-1/2}  L_{\rho ,k} }{H_{\rho, 2, i-2}H_{\rho,3,k+2-i}  }\\
		&\qquad\qquad\qquad\qquad\times\big( H_{\rho,2,i-2} \norm{   \partial_y^{i} v}_{L^2 } \big)\big(H_{\rho,3, k+2-i} \norm{ \comi y^{\ell} \partial_y^{k+2-i}u}_{L^\infty}\big).	\end{aligned}
\end{eqnarray*}
For the term $p_k$, we first note the following  estimate (see Appendix \ref{sec:app} for its proof) that
\begin{equation}\label{fe3}
\forall\ 1\leq i \leq [(k+1)/2],\quad   \frac{(k+1)!}{i!(k+1-i)!}  \frac{ (k+1)^{-{1\over2}} L_{\rho ,k} }{H_{\rho,4,i-1}L_{\rho ,k+1-i}  }   \leq C \frac{(k+2-i)^{1\over2}}{ i +1}.
\end{equation}
Following an argument  similar to  \eqref{YE}, we have 
\begin{equation*}
\begin{aligned}
	& \sum_{k=0}^{+\infty} p_k^2  \leq  C\sum_{k=0}^{+\infty} \bigg[\sum_{i=1}^{ k+1}
		 \frac{  H_{\rho,4 ,i-1} \norm{ \partial_y^{i} v}_{L^\infty }  }{i+1}  (k+2-i)^{1\over 2}  L_{\rho , k+1-i} \norm{ \comi y^{\ell}  \partial_y^{k+2-i}u}_{L^2} \bigg]^2\\
		 &\leq     C\bigg( \sum_{i=1}^{+\infty}
		 \frac{  H_{\rho,4 ,i-1} \norm{ \partial_y^{i} v}_{L^\infty }  }{i+1} \bigg)^2 \sum_{i=0}^{+\infty}
		\inner{i+1} L_{\rho , i}^2 \norm{ \comi y^{\ell}  \partial_y^{i+1}u}_{L^2}^2\\
		& \leq C |\vec a|_{Y_\rho}^2 \sum_{i=1}^{+\infty} H_{\rho,4 ,i-1} ^2 \norm{ \partial_y^{i} v}_{L^\infty } ^2,
		 \end{aligned}
\end{equation*}
where we have used   \eqref{dis} in the second inequality.   Moreover, using   the Sobolev embedding inequality and Hardy's inequality, we obtain
\begin{multline*}
	\sum_{i=1}^{+\infty} H_{\rho,4,i-1} ^2 \norm{ \partial_y^{i} v}_{L^\infty } ^2 \leq  		  H_{\rho,4,0}^2 \norm{\partial_x u}_{L^\infty }^2 +    \sum_{i=2}^{+\infty}
		  H_{\rho,4,i-1}^2 \norm{\partial_x \partial_y^{i-1} u}_{L^\infty }  ^2 \\
		    \leq C   	  \norm{\comi y^\ell \partial_x \partial_yu}_{L_x^\infty L_y^2 }^2 +C |\vec a|_{X_{\rho}}^2\leq C |\vec a|_{X_{\rho}}^2.
\end{multline*}
Combining the above estimates gives
\begin{equation}\label{ak}
	\sum_{k=0}^{+\infty} p_k^2\leq C  |\vec a|_{X_{\rho}}^2  |\vec a|_{Y_{\rho}}^2.
\end{equation}
Similarly, by using the estimate (see Appendix \ref{sec:app} for its proof)
\begin{equation}\label{fe4}
	\forall\  [(k+1)/2]< i\leq k+1,\ \	\frac{(k+1)!}{i!(k+1-i)!}  \frac{(k+1)^{-{1\over2}} L_{\rho ,k} }{H_{\rho, 2, i-2}H_{\rho,3 ,k+2-i}  }\leq  \frac{C }{k+3-i},
\end{equation}
and observing $|\vec a|_{X_\rho}\leq |\vec a|_{Y_\rho},$ we have
 \begin{align*}
	  \sum_{k=0}^{+\infty} q_k^2 \leq C\bigg( \sum_{i=0}^{+\infty}
		\frac{H_{\rho, 3 , i} \norm{  \partial_y^{i}u}_{L^\infty}}{i+1} \bigg)^2 \sum_{i=0}^{+\infty}
		  H_{\rho, 2,   i}^2 \norm{ \partial_x \partial_y^{i+1} u}_{L^2}^2
		   \leq   C|\vec a|_{X_\rho}^2|\vec a|_{Y_\rho}^2.
\end{align*}
This together with   \eqref{ak} and \eqref{ambm} yields
\begin{equation*}
	\sum_{k=0}^{+\infty} \bigg[ (k+1)^{-1/2}  L_{\rho , k}   \sum_{ i=1}^{k+1} {{k+1}\choose i}  \norm{\comi y^\ell (  \partial_y^i v) \partial_y^{k+2-i} u}_{L^2}\bigg]^2\leq C|\vec a|_{X_{\rho}}^2 |\vec a|_{Y_{\rho}}^2.
\end{equation*}
Similarly,
\begin{equation*}
	\sum_{k=0}^{+\infty} \bigg[  (k+1)^{-1/2}  L_{\rho , k}    \sum_{ i=1}^{k+1} {{k+1}\choose i}   \norm{\comi y^\ell(  \partial_y^i u)\partial_x \partial_y^{k+1-i} u}_{L^2}\bigg]^2 \leq C|\vec a|_{X_{\rho}}^2 |\vec a|_{Y_{\rho}}^2.
\end{equation*}
Finally, by  substituting the above two estimates into \eqref{rk} and  using \eqref{XY},
we obtain
\begin{align*}
	\sum_{k=0}^{+\infty}\Big[ \inner{k+1}^{-1/2} L_{\rho , k}     \norm{ \mathcal R_k}_{L^2}\Big]^2\leq  C\big(1+|\vec a|_{X_{\rho}}^2\big) |\vec a|_{Y_{\rho}}^2. \end{align*}
This with \eqref{asak} yields  the desired
  assertion \eqref{esterror}. The proof of the claim\eqref{esterror} is completed.
  \end{proof}

\subsection{Tangential derivatives}\label{subsec:tan}

 In this subsection, we consider  the  tangential derivatives of $\varphi$ and $\partial_yu$.

\begin{lemma}\label{lem:tan}
Under the assumptions of Theorem \ref{thm:apriori}, we have,
\begin{equation*}
	\begin{aligned}
 	&\frac12\frac{d}{dt} \sum_{m=0}^{+\infty}(m+1)^2 N_{\rho,m+1}^2 \Big( \norm{ \comi y^\ell  \partial_x^m  \varphi }_{L^2}^2+   \norm{\comi y^\ell  \partial_x^m\partial_y u }_{L^2}^2\Big)   \\
	 &\leq C  \big(1+|\vec a|_{X_\rho}^4 \big)  |\vec a|_{Y_\rho}^2 -\mu \sum_{m=0}^{+\infty} (m+1)^3 N_{\rho,m+1}^2 \big( \norm{ \comi y^\ell  \partial_x^m  \varphi }_{L^2}^2+   \norm{\comi y^\ell  \partial_x^m\partial_y u }_{L^2}^2\big).
	\end{aligned}
\end{equation*}
Recall $N_{\rho,m}$ is defined by \eqref{nl}.	
\end{lemma}

\begin{proof}
We apply   $\comi y^\ell \partial_x^m\partial_y$ and $\comi y^\ell \partial_x^m $ to the first and   second equations in \eqref{neweq} respectively to have 
\begin{equation*}
	\left\{
	\begin{aligned}
		&\big(\partial_t+u\partial_x+v\partial_y\big) \comi y^\ell  \partial_x^m\partial_y u =\comi y^\ell  \partial_x^m \partial_y\varphi +\mathcal P_m,\\
		&\big(\partial_t+1\big) \comi y^\ell  \partial_x^m\varphi=\comi y^\ell  \partial_x^m\partial_y^2u,
	\end{aligned}
	\right.
\end{equation*}
where
\begin{equation}\label{depm}
	\mathcal P_m=v(\partial_y\comi y^\ell )\partial_x^m \partial_y u -\sum_{j=1}^m {m\choose j}\comi y^\ell\Big[(\partial_x^j u)\partial_x^{m-j+1}\partial_yu+ (\partial_x^j v)\partial_x^{m-j}\partial_y^2u\Big].
\end{equation}
  Following a similar argument as  the proof of Lemma \ref{lem:nor}, by  observing that
  \begin{multline*}
    	 \inner{ \comi y^{\ell} \partial_x^m\partial_y^2u, \ \comi y^{\ell}\partial_x^m\varphi}_{L^2}\\
    	 =-\inner{\comi y^{\ell}\partial_x^m\partial_yu, \ \comi y^{\ell}\partial_x^m\partial_y\varphi}_{L^2}-\inner{(\partial_y \comi y^{2\ell})\partial_x^m\partial_yu, \ \partial_x^m\varphi}_{L^2},
  \end{multline*}
  we get
\begin{multline*}
	 \frac12\frac{d}{dt} \Big(  \norm{ \comi y^\ell  \partial_x^m  \varphi }_{L^2}^2+  \norm{\comi y^\ell  \partial_x^m\partial_y u }_{L^2}^2\Big)+\norm{ \comi y^\ell  \partial_x^m  \varphi }_{L^2}^2 \\
	  =   -\inner{(\partial_y \comi y^{2\ell})\partial_x^m\partial_yu, \ \partial_x^m\varphi}_{L^2}+ \inner{\mathcal P_m,\  \comi y^\ell \partial_x^m\partial_yu  }_{L^2}.
\end{multline*}
Thus, by using \eqref{dev:rho} and the  fact that $-\mu(m+2)\leq -\mu (m+1),$ we have 
\begin{equation*}
\begin{aligned}
 	&\frac12\frac{d}{dt} \sum_{m=0}^{+\infty}(m+1)^2 N_{\rho,m+1}^2 \Big( \norm{ \comi y^\ell  \partial_x^m  \varphi }_{L^2}^2+   \norm{\comi y^\ell  \partial_x^m\partial_y u }_{L^2}^2\Big)   \\
	 &\leq -\mu \sum_{m=0}^{+\infty} (m+1)^3 N_{\rho,m+1}^2 \Big( \norm{ \comi y^\ell  \partial_x^m  \varphi }_{L^2}^2+   \norm{\comi y^\ell  \partial_x^m\partial_y u }_{L^2}^2\Big)\\
	 &\quad+C  |\vec a|_{X_\rho}^2 +\sum_{m=0}^{+\infty} (m+1)^2 N_{\rho,m+1}^2 \inner{\mathcal P_m,\  \comi y^\ell \partial_x^m\partial_yu  }_{L^2}.
	\end{aligned}
\end{equation*}
It remains to estimate the last term on the right-hand side. Similar to \eqref{asak}, it holds that 
\begin{multline}\label{m1pm}
	\sum_{m=0}^{+\infty} (m+1)^2 N_{\rho,m+1}^2 \inner{\mathcal P_m,\  \comi y^\ell \partial_x^m\partial_yu  }_{L^2}\\
	\leq  |\vec a|_{Y_\rho}^2+ \sum_{m=0}^{+\infty}\Big[ (m+1)^{1/2}   N_{\rho , m+1}     \norm{ \mathcal P_m }_{L^2}\Big]^2.
\end{multline}
Then  Lemma \ref{lem:tan} holds  by the above inequalities if 
\begin{equation}\label{erpm}
	\sum_{m=0}^{+\infty}\Big[ (m+1)^{1/2}   N_{\rho , m+1}     \norm{ \mathcal P_m }_{L^2}\Big]^2\leq C\big(1+|\vec a|_{X_\rho}^4 \big)  |\vec a|_{Y_\rho}^2.
\end{equation}
We now turn to prove  \eqref{erpm}. In view of \eqref{depm}, we use  the fact that $\norm{v}_{L_y^\infty}\leq C\norm{\comi y^\ell \partial_x\partial_yu}_{L_y^2}$ by Hardy's inequality to obtain
\begin{equation}\label{mpm}
\begin{aligned}
	& \sum_{m=0}^{+\infty}\Big[ (m+1)^{1/2}   N_{\rho , m+1}     \norm{ \mathcal P_m }_{L^2}\Big]^2\\
	& \leq C|\vec a|_{X_\rho}^3+ \sum_{m=0}^{+\infty} \bigg[    (m+1)^{1/2}   N_{\rho , m+1}    \sum_{ j=1}^{m} {m\choose j}   \norm{\comi y^\ell(  \partial_x^j u)\partial_x^{m-j+1} \partial_y u}_{L^2}\bigg]^2  \\
&
\quad +\sum_{m=0}^{+\infty} \bigg[   (m+1)^{1/2}  N_{\rho , m+1}   \sum_{ j=1}^{m} {{m}\choose j}  \norm{\comi y^\ell (  \partial_x^j v) \partial_x^{m-j}\partial_y^{2} u}_{L^2}\bigg]^2.
 \end{aligned}
\end{equation}
For the last term above inequality,  as  \eqref{ambm}, we write
\begin{equation} \label{nm}
	 (m+1)^{1/2} N_{\rho , m+1}   \sum_{ j=1}^{m} {{m}\choose j}  \norm{\comi y^\ell (  \partial_x^j v) \partial_x^{m-j}\partial_y^{2} u}_{L^2}=  A_m+B_m,
\end{equation}
with
\begin{eqnarray*}
	\begin{aligned}
		 &A_m=    \sum_{j=1}^{[m/2]}   \frac{m!}{j!(m-j)!}  \frac{(m+1)^{1/2}   N_{\rho ,m+1} }{ N_{\rho , j+3}H_{\rho ,m-j+1,1}  }\\
		&\qquad\qquad\qquad\qquad\times\big( N_{\rho ,j+3} \norm{ \partial_x^{j} v}_{L^\infty } \big)\big( H_{\rho, m-j+1, 1} \norm{ \comi y^{\ell} \partial_x^{m-j} \partial_y^{2} u}_{L^2}\big),\\
			&B_m =    \sum_{j=[m/2]+1}^{m} \frac{m!}{j!(m-j)!}   \frac{ (m+1)^{1/2}  N_{\rho ,m+1} }{  N_{\rho ,j+1}H_{\rho ,m-j+3,1}  }\\
		& \qquad\qquad \qquad \times\big(  N_{\rho , j+1} \norm{ \partial_x^{j} v}_{L_x^2L_y^\infty} \big) \big( H_{\rho , m-j+3,1} \norm{ \comi y^{\ell}    \partial_x^{m-j} \partial_y^{2}u}_{L_x^\infty L_y^2}\big).
	\end{aligned}
\end{eqnarray*}
Moreover, by using the estimate (see Appendix \ref{sec:app} for its proof)
\begin{equation}\label{fe6}
	\forall\ 1\leq j\leq [m/2],\quad  \frac{m!}{j!(m-j)!} \frac{  (m+1)^{1\over2}  N_{\rho ,m+1} }{ N_{\rho , j+3}H_{\rho ,m-j+1,1}  }\leq  \frac{C(m-j+1)^{3\over2}}{j+1},
\end{equation}
we follow  a similar argument as that after \eqref{ambm} to conclude
  that
\begin{align*}
	\sum_{m=0}^{+\infty} A_m^2  &\leq C \Big(\sum_{j=0}^{+\infty}\frac{ N_{\rho,j+3}  \norm{ \partial_x^{j} v}_{L^\infty }}{ j+1} \Big)^2\sum_{j=0}^{+\infty} (j+1) ^3 H_{\rho , j+1,1}^2 \norm{ \comi y^\ell \partial_x^{j} \partial_y^2u}_{L^2}^2\\
	&\leq C\inner{1+|\vec a|_{X_\rho}^2} |\vec a|_{X_\rho}^2 |\vec a|_{Y_\rho}^2    \leq   C\big(1+|\vec a|_{X_\rho}^4 \big)  |\vec a|_{Y_\rho}^2,
\end{align*}
where we have used Corollary \ref{cor:v} in the second inequality. On the other hand, we note that
\begin{equation}\label{fe5}
\forall\  [m/2]+1\leq j\leq m,\quad	\frac{m!}{j!(m-j)!}  \frac{ (m+1)^{1\over2}    N_{\rho ,m+1} }{ N_{\rho ,j+1}H_{\rho ,m-j+3,1}  }\leq C \frac{(j+1)^{1/2}}{m-j+1},
\end{equation}
where its proof is given in Appendix \ref{sec:app}. Thus, following a similar argument as that after \eqref{ambm} and using Corollary \ref{cor:v}, we conclude
\begin{align*}
	\sum_{m=0}^{+\infty} B_m^2 &\leq C \Big(\sum_{j=0}^{+\infty}\frac{ H_{\rho,j+3,1}  \norm{ \comi y^{\ell}    \partial_x^{j} \partial_y^{2}u}_{L_x^\infty L_y^2}}{ j+1} \Big)^2\sum_{j=0}^{+\infty} (j+1) N_{\rho , j+1}^2 \norm{ \partial_x^{j} v}_{L_x^2L_y^\infty}^2\\
	&\leq C|\vec a|_{X_\rho}^2\inner{1+|\vec a|_{X_\rho}^2}  |\vec a|_{Y_\rho}^2   \leq   C\big(1+|\vec a|_{X_\rho}^4 \big)  |\vec a|_{Y_\rho}^2.
		\end{align*}
 As a result, we combine the above estimates with \eqref{nm} to obtain
\begin{equation}\label{tanest}
	\sum_{m=0}^{+\infty} \bigg[	 (m+1)^{1\over 2} N_{\rho , m+1}   \sum_{ j=1}^{m} {{m}\choose j}  \norm{\comi y^\ell (  \partial_x^j v) \partial_x^{m-j}\partial_y^{2} u}_{L^2}\bigg]^2 \leq   C\big(1+|\vec a|_{X_\rho}^4 \big)  |\vec a|_{Y_\rho}^2.
\end{equation}
Similar argument also yields
\begin{align*}
	\sum_{m=0}^{+\infty} \bigg[  (m+1)^{1\over2}   N_{\rho , m+1}    \sum_{ j=1}^{m} {m\choose j}   \norm{\comi y^\ell(  \partial_x^j u)\partial_x^{m-j+1} \partial_y u}_{L^2}\bigg]^2 \leq   C (1+|\vec a|_{X_\rho}^4 )  |\vec a|_{Y_\rho}^2.
\end{align*}
Then substituting the above two inequalities into \eqref{mpm}, the estimate \eqref{erpm} follows. The proof of the lemma  is completed.
\end{proof}

 \subsection{Proof of Proposition \ref{prp:varu}} We will follow the argument used in the 
  Subsections \ref{subsec:nor} and \ref{subsec:tan}  to derive the estimate on   the mixed derivatives. For this, we apply  $ \comi y^\ell \partial_x^m \partial_y^{k+1} $   and  $\comi y^\ell \partial_x^m \partial_y^{k} $ to the first and the second equations in \eqref{neweq} to obtain 
 \begin{equation*}
 	\left\{\begin{aligned}
 		&\big(\partial_t+u\partial_x+v\partial_y\big)   \comi y^\ell \partial_x^m \partial_y^{k+1}   u = \comi y^\ell \partial_x^m \partial_y^{k+1}   \varphi  +\mathcal Q_{m,k},
\\
&\big(\partial_t+1\big)   \comi y^\ell \partial_x^m \partial_y^{k} \varphi =  \comi y^\ell \partial_x^m \partial_y^{k} \partial_y^2u,
 	\end{aligned}
 	\right.
 \end{equation*}
where
\begin{equation}\label{qmk}
\begin{aligned}
	 \mathcal Q_{m,k}=&v (\partial_y  \comi y^\ell )\partial_x^m \partial_y^{k+1}   u -\sum_{ i+j\geq 1 }  {m\choose j}{{k+1}\choose i}\comi y^\ell  (\partial_x^j \partial_y^i u)\partial_x^{m-j+1}\partial_y^{k+1-i}u  \\
	&\quad-\sum_{\stackrel{i+j\geq 1}{j\leq m, i\leq k+1} }  {m\choose j}{{k+1}\choose i}\comi y^\ell     (\partial_x^j \partial_y^i v)\partial_x^{m-j}\partial_y^{k+2-i} u.
	\end{aligned}
\end{equation}
Thus, we follow  a similar argument as   in the proof of \eqref{noebou}   and observe the fact that
\begin{equation*}
	\frac12\frac{d}{dt}\inner{H_{\rho,m+1,k}^2 }=-\mu(m+k+2)H_{\rho,m+1,k}^2 \leq -\mu(m+k+1)H_{\rho,m+1,k}^2,
\end{equation*}
to obtain
\begin{equation}\label{taestimate}
\begin{aligned}
	&\frac12\frac{d}{dt} \sum_{m=0}^{+\infty}\sum_{k=0}^{+\infty} (m+1)^2 H_{\rho,m+1,k}^2  \Big( \norm{ \comi y^\ell  \partial_x^m \partial_y^k  \varphi }_{L^2}^2+   \norm{\comi y^\ell  \partial_x^m \partial_y^{k+1} u }_{L^2}^2\Big) \\
	 &\leq C  |\vec a|_{Y_\rho}^2 +\sum_{m=0}^{+\infty}\sum_{k=0}^{+\infty} (m+1)^2 H_{\rho,m+1,k}^2   \inner{\mathcal Q_{m,k},\  \comi y^\ell \partial_x^m\partial_y^{k+1}u  }_{L^2} \\
	 &\   -\mu \sum_{m,k}  \inner{m+k+1}  (m+1)^2  H_{\rho,m+1,k}^2 ( \norm{ \comi y^\ell  \partial_x^m \partial_y^k \varphi }_{L^2}^2+   \norm{\comi y^\ell  \partial_x^m\partial_y^{k+1} u }_{L^2}^2 ).
\end{aligned}
\end{equation}
Similar to \eqref{asak} and \eqref{m1pm}, by recalling $\mathcal Q_{m,k}$ is given by \eqref{qmk}, we have
\begin{equation}\label{s1s2}
	\sum_{m,k\geq0}  (m+1)^2 H_{\rho,m+1,k}^2  \big(\mathcal Q_{m,k},\  \comi y^\ell \partial_x^m\partial_y^{k+1}u  \big)_{L^2}\leq C |\vec a|_{X_\rho}^3+|\vec a|_{Y_\rho}^2+ S_1+S_2,
\end{equation}
with
\begin{align*}
	S_1=&\sum_{m,k\geq0} \bigg[ (m+k+1)^{-1/2}(m+1)   H_{\rho , m+1,k}   \\
	&\qquad\quad\quad  \times  \sum_{ \stackrel{i+j\geq1}{j\leq m, i\leq k+1}} {m\choose j} {{k+1}\choose i}  \norm{\comi y^\ell  (\partial_x^j \partial_y^i u)\partial_x^{m-j+1}\partial_y^{k+1-i}u}_{L^2}\bigg]^2,\\
	S_2=&\sum_{m,k\geq0} \bigg[  (m+k+1)^{-1/2}(m+1)  H_{\rho , m+1,k}   \\
	&\qquad\qquad\qquad \times  \sum_{ \stackrel{i+j\geq1}{j\leq m, i\leq k+1}} {m\choose j} {{k+1}\choose i}  \norm{\comi y^\ell     (\partial_x^j \partial_y^i v)\partial_x^{m-j}\partial_y^{k+2-i} u}_{L^2}\bigg]^2.
\end{align*}
To estimate $S_2$, we write
\begin{align*}
&	(m+k+1)^{-1/2}(m+1)  H_{\rho , m+1,k}    \\
&\qquad\qquad  \times  \sum_{ \stackrel{i+j\geq1}{j\leq m, i\leq k+1}} {m\choose j} {{k+1}\choose i}  \norm{\comi y^\ell     (\partial_x^j \partial_y^i v)\partial_x^{m-j}\partial_y^{k+2-i} u}_{L^2} \\
&\leq r_{m,k}+p_{m,k}+q_{m,k}
\end{align*}
with
\begin{align*}
&	r_{m,k} = (m+k+1)^{-1/2}(m+1)  H_{\rho , m+1,k}   \sum_{ j=1}^{ m } {m\choose j}   \norm{\comi y^\ell     (\partial_x^j   v)\partial_x^{m-j}\partial_y^{k+2} u}_{L^2}\\
	&+ (m+k+1)^{-{1\over2}}(m+1)(k+1)  H_{\rho , m+1,k}   \sum_{ j=0}^{ m } {m\choose j}   \norm{\comi y^\ell     (\partial_x^j \partial_y  v)\partial_x^{m-j}\partial_y^{k+1} u}_{L^2},
\end{align*}
and
\begin{align*}
	p_{m,k}&=\sum_{ \stackrel{  i+j\leq [(m+k+1)/2] }{j\leq m,\  2\leq i\leq  k+1}}{m\choose j} {{k+1}\choose i} \frac{ (m+k+1)^{-1/2}(m+1)  H_{\rho , m+1,k}}{H_{\rho, j+4,i-1}H_{\rho, m-j+1,k+1-i}} \\
	&\qquad\times \big( H_{\rho, j+4,i-1}\norm{ \partial_x^j \partial_y^i v}_{L^\infty}\big)\big(H_{\rho, m-j+1,k+1-i}\norm{\comi y^\ell \partial_x^{m-j}\partial_y^{k+2-i} u}_{L^2} \big),\\
	q_{m,k}&=\sum_{ \stackrel{ i+j \geq [(m+k+1)/2]+1 }{j\leq m,\  2\leq i\leq  k+1}}{m\choose j} {{k+1}\choose i}  \frac{(m+k+1)^{-1/2}(m+1) H_{\rho , m+1,k}}{H_{\rho, j+2,i-2}H_{\rho, m-j+3,k+2-i}} \\
	&\quad\times \big( H_{\rho, j+2,i-2}\norm{ \partial_x^j \partial_y^i v}_{L^2}\big)\big(H_{\rho, m-j+3,k+2-i}\norm{\comi y^\ell \partial_x^{m-j}\partial_y^{k+2-i} u}_{L^\infty   } \big).
\end{align*}
Like the proof of \eqref{tanest},  we can obtain
\begin{equation*}
	 \sum_{m,k\geq0} r_{m,k}^2  \leq  C \big(1+|\vec a|_{X_\rho}^4 \big)  |\vec a|_{Y_\rho}^2.
\end{equation*}
Moreover,  if $1\leq i+j\leq [(m+k+1)/2] ,$ then
\begin{equation}\label{fe10}
\begin{aligned}
&	{m\choose j} {{k+1}\choose i}  \frac{(m+k+1)^{-{1\over 2}}(m+1) H_{\rho , m+1,k}}{H_{\rho, j+4,i-1}H_{\rho, m-j+1,k+1-i}}\\
	&\leq ( j+4) (m-j+1)\frac{1}{ (i+j+1)^2}    (m+k-i-j+2)^{1/2},
	\end{aligned}
\end{equation}
where its proof is given inAppendix \ref{sec:app}.   Then following  the argument after \eqref{ambm} and \eqref{nm}, we obtain
\begin{align*}
 \sum_{m,k\geq0} p_{m,k}^2 &\leq  \bigg(\sum_{j\geq 0,i\geq 2} \frac{(j+4)H_{\rho,j+4,i-1} \norm{ \partial_x^j \partial_y^{i} v}_{L^\infty}}{(i+j+1)^2}\bigg)^2\\
 &\qquad\qquad\qquad\quad\times \sum_{i,j\geq 0} (i+j+1) (j+1)^2 H_{\rho, j+1,i}\norm{\comi y^\ell \partial_x^{j}\partial_y^{i} \partial_y u}_{L^2}^2\\
	&  \leq  C \big(1+|\vec a|_{X_\rho}^2 \big)  |\vec a|_{Y_\rho}^2.
\end{align*}
Similarly,  by using
\begin{equation}\label{laetimate}
{m\choose j} {{k+1}\choose i}  \frac{(m+k+1)^{-{1\over 2}}(m+1) H_{\rho , m+1,k}}{H_{\rho, j+2,i-2}H_{\rho, m-j+3,k+2-i}}\leq C \frac{ j+1   }{(m+k-i-j+2)^2 }
\end{equation}
for any pair $(i,j)$ with $   [(m+k+1)/2]\leq  i+j \leq m+k +1$  (see Appendix \ref{sec:app} for its proof), we have 
\begin{align*}
	\sum_{m,k\geq0} q_{m,k}^2\leq C\big(1+|\vec a|_{X_\rho}^2 \big)  |\vec a|_{Y_\rho}^2.
\end{align*}
Thus combining  the above estimates gives
\begin{align*}
	S_2\leq C \sum_{m,k\geq0} (r_{m,k}^2+p_{m,k}^2+q_{m,k}^2) \leq   C\big(1+|\vec a|_{X_\rho}^4 \big)  |\vec a|_{Y_\rho}^2. 
\end{align*}
Similarly,
\begin{equation*}
	S_1\leq C\big(1+|\vec a|_{X_\rho}^4 \big)  |\vec a|_{Y_\rho}^2.
\end{equation*}
This with \eqref{s1s2} and \eqref{taestimate} yields
  the statement in Proposition \ref{prp:varu}. The proof of  Proposition \ref{prp:varu} is  completed.

\section{Proof of the \emph{a priori} estimate} \label{sec:comple}

We now study the Gevrey estimate on the auxiliary functions $\mathcal U$ and $\mathcal \lambda$ defined in \eqref{mau}
	and \eqref{ma} to complete the proof of Theorem \ref{thm:apriori}.

\begin{proposition}\label{prp:ulam}
Under the assumptions of Theorem \ref{thm:apriori}, we have
	\begin{multline*}
		 \frac12\frac{d}{dt} \sum_{m=0}^{+\infty} \inner{ N_{\rho,m+1}^2\norm{  \partial_x^m  \mathcal U }_{L^2}^2+ (m+1) N_{\rho,m+1}^2 \norm{  \comi y^{\ell-1} \partial_x^m  \lambda }_{L^2}^2 } \\  \leq -\mu \sum_{m=0}^{+\infty}   \Big[ (m+1)N_{\rho,m+1}^2\norm{  \partial_x^m  \mathcal U }_{L^2}^2+(m+1)^2N_{\rho,m+1}^2 \norm{ \comi y^{\ell-1}  \partial_x^m  \lambda}_{L^2}^2 \Big]\\
		 +C (1+|\vec a|_{X_\rho}^4)|\vec a|_{Y_\rho}^2.
	\end{multline*}
\end{proposition}

\begin{proof}
  It follows from \eqref{ymau} that
\begin{align*}
	\inner{\partial_t+u\partial_x+v\partial_y} \partial_x^m \mathcal U
= \partial_x^{m+1} \lambda +\partial_x^m\Big[ (\partial_x\partial_yu)\int_0^y\mathcal U(t,x,\tilde y)  d\tilde y+(\partial_xu)\mathcal U \Big].
\end{align*}
This, with the fact that $\frac{1}{2}\frac{d}{dt}N_{\rho,m+1}^2\leq -\mu(m+1) N_{\rho,m+1}^2$ and
\begin{align*}
	\sum_{m=0}^{+\infty} (m+1)  N_{\rho,m+1}^2 \norm{  \partial_x^m  \mathcal U }_{L^2}^2 \leq |\vec a|_{Y_\rho}^2,
\end{align*}
yields that
\begin{align*}
	&\frac12\frac{d}{dt} \sum_{m=0}^{+\infty} N_{\rho,m+1}^2 \norm{  \partial_x^m  \mathcal U }_{L^2}^2    \leq -\mu \sum_{m=0}^{+\infty} (m+1)  N_{\rho,m+1}^2  \norm{  \partial_x^m  \mathcal U }_{L^2}^2\\
	 &\qquad\qquad+   |\vec a|_{Y_\rho}^2 +\sum_{m=0}^{+\infty} (m+1)^{-1} N_{\rho,m+1}^2\norm{\partial_x^{m+1}\lambda}_{L^2}^2\\
	 &\qquad\qquad+\sum_{m=0}^{+\infty}   (m+1)^{-1}N_{\rho,m+1}^2 \big\|\partial_x^m\big[ (\partial_x\partial_yu)\int_0^y\mathcal U(t,x,\tilde y)  d\tilde y+(\partial_xu)\mathcal U \big]\big\|_{L^2}^2.
\end{align*}
By definition \eqref{defY}  of $|\vec a|_{Y_\rho}$ and the fact that  $N_{\rho,m+1}/N_{\rho,m+2}\leq C (m+1)^{3/2}$,  we have
\begin{multline*}
	\sum_{m=0}^{+\infty} (m+1)^{-1} N_{\rho,m+1}^2\norm{\partial_x^{m+1}\lambda}_{L^2}^2 \leq C \sum_{m=0}^{+\infty}  (m+1)^2   N_{\rho,m+2}^2\norm{\partial_x^{m+1}\lambda}_{L^2}^2\\
	\leq C\sum_{m=0}^{+\infty}   (m+1)^2  N_{\rho,m+1}^2\norm{\partial_x^{m}\lambda}_{L^2}^2\leq C |\vec a|_{Y_\rho}^2.
\end{multline*}
Following a similar argument as in the proof of Lemmas  \ref{lem:tang} and \ref{lem:tan},  we conclude
\begin{multline*}
	\sum_{m=0}^{+\infty}(m+1)^{-1}N_{\rho,m+1}^2 \big\|\partial_x^m\big[ (\partial_x\partial_yu)\int_0^y\mathcal U(t,x,\tilde y)  d\tilde y+(\partial_xu)\mathcal U \big]\big\|_{L^2}^2\\
	\leq C \big(1+|\vec a|_{X_\rho}^4 \big)|\vec a|_{Y_\rho}^2.
\end{multline*}
As a result, combining the above estimates yields
\begin{equation}\label{muc}
	 \frac12\frac{d}{dt} \sum_{m=0}^{+\infty} N_{\rho,m+1}^2 \norm{  \partial_x^m  \mathcal U }_{L^2}^2    \leq C (1+|\vec a|_{X_\rho}^4)|\vec a|_{Y_\rho}^2 -\mu \sum_{m=0}^{+\infty} (m+1)  N_{\rho,m+1}^2  \norm{  \partial_x^m  \mathcal U }_{L^2}^2.
\end{equation}
It remains to estimate $\lambda$, and  it follows from \eqref{lateq} that
\begin{multline*}
	\big(\partial_t  +u\partial_x +v\partial_y  \big)\comi y^{\ell-1} \partial_x^m\lambda
	= v(\partial_y\comi y^{\ell-1}) \partial_x^m\lambda\\
	+\comi y^{\ell-1} \partial_x^{m+1}\varphi  -\comi y^{\ell-1} \partial_x^m\Big[(\partial_xu)\partial_xu+  (\partial_y\varphi)  \int_0^y\mathcal U d\tilde y\Big].
\end{multline*}
This with the fact that
 $$\sum_{m=0}^{+\infty} (m+1)^2 N_{\rho,m+1}^2 \norm{ \comi y^{\ell-1}  \partial_x^m  \lambda}_{L^2}^2\leq |\vec a|_{Y_\rho}^2,$$
 yields
\begin{align*}
	&\frac{1}{2}\frac{d}{dt} \sum_{m=0}^{+\infty} (m+1) N_{\rho,m+1}^2 \norm{  \comi y^{\ell-1} \partial_x^m  \lambda }_{L^2}^2 \\
	&   \leq -\mu \sum_{m=0}^{+\infty}  (m+1)^2 N_{\rho,m+1}^2  \norm{ \comi y^{\ell-1}  \partial_x^m  \lambda}_{L^2}^2+ C|\vec a|_{X_\rho}^3 +  |\vec a|_{Y_\rho}^2  \\
	 &\quad +\sum_{m=0}^{+\infty}  N_{\rho,m+1}^2\norm{\comi y^{\ell-1}\partial_x^{m+1}\varphi}_{L^2}^2\\
	 &\quad +\sum_{m=0}^{+\infty}   N_{\rho,m+1}^2 \big\|\comi y^{\ell-1}\partial_x^m\big[(\partial_xu)\partial_xu+  (\partial_y\varphi)  \int_0^y\mathcal U d\tilde y \big]\big\|_{L^2}^2.
\end{align*}
Note that 
\begin{multline*}
	\sum_{m=0}^{+\infty}  N_{\rho,m+1}^2\norm{\comi y^{\ell-1}\partial_x^{m+1}\varphi}_{L^2}^2\\
	\leq C \sum_{m=0}^{+\infty}  (m+2)^3 N_{\rho,m+2}^2\norm{\comi y^{\ell}\partial_x^{m+1}\varphi}_{L^2}^2\leq C|\vec a|_{Y_\rho}^2.
\end{multline*}
By a similar argument as the proof of Lemma \ref{lem:tang},  we have 
\begin{align*}
	\sum_{m=0}^{+\infty}   N_{\rho,m+1}^2 \big\|\comi y^{\ell-1}\partial_x^m\big[(\partial_xu)\partial_xu+  (\partial_y\varphi)  \int_0^y\mathcal U d\tilde y \big]\big\|_{L^2}^2 \leq C\big(1+|\vec a|_{X_\rho}^4 \big)|\vec a|_{Y_\rho}^2.
\end{align*}
Consequently, combining the above inequalities yields
\begin{multline*}
	\frac{1}{2}\frac{d}{dt} \sum_{m=0}^{+\infty} (m+1) N_{\rho,m+1}^2 \norm{  \comi y^{\ell-1} \partial_x^m  \lambda }_{L^2}^2 \\
	    \leq -\mu \sum_{m=0}^{+\infty}  (m+1) ^2 N_{\rho,m+1}^2  \norm{ \comi y^{\ell-1}  \partial_x^m  \lambda}_{L^2}^2+   C\big(1+|\vec a|_{X_\rho}^4 \big) |\vec a|_{Y_\rho}^2.
\end{multline*}
This and  \eqref{muc} yields the statement in Proposition \ref{prp:ulam} so that the proof of Proposition \ref{prp:ulam} is   completed.
\end{proof}

\begin{proof}
	[Proof of Theorem \ref{thm:apriori}] With the two estimates in Propositions \ref{prp:varu} and \ref{prp:ulam} and by noting the definitions \eqref{defX} and \eqref{defY} of $|\vec a|_{X_\rho}$ and $|\vec a|_{Y_\rho}$, we have
	\begin{equation}\label{laet}
		\frac12\frac{d}{dt}|\vec a|_{X_\rho}^2\leq \inner{-\mu+C +C |\vec a|_{X_\rho}^4}|\vec a|_{Y_\rho}^2.
	\end{equation}
	Now we choose $\mu$ large enough such that
	\begin{equation}\label{mu++}
		\mu \geq \frac12+  C + C     (2C_0)^4 \Big(\norm{u_0}_{G^{3/2,1}_{2\rho_0,\ell}}+\norm{u_1}_{G^{3/2,1}_{2\rho_0,\ell+1}}\Big)^4,
	\end{equation}
	where $C_0$ is the constant in \eqref{init}.
	Then under assumption \eqref{bootass}, \eqref{laet} and \eqref{mu++} yield
	\begin{align*}
	\forall\ t\in [0, T], \quad 		\frac12\frac{d}{dt}|\vec a(t)|_{X_\rho}^2\leq -\frac12 |\vec a(t)|_{Y_\rho}^2.
	\end{align*}
	By  \eqref{init}, we have 
		\begin{align*}
	  	\sup_{0\leq t\leq T}\abs{\vec a(t)}_{X_\rho} +\Big(\int_0^T |\vec a(t)|_{Y_\rho}^2 dt \Big)^{1\over2} \leq \abs{\vec a(0)}_{X_\rho}\leq C_0\big(\norm{u_0}_{G^{3/2,1}_{2\rho_0,\ell}}+\norm{u_1}_{G^{3/2,1}_{2\rho_0,\ell+1}}\big).
	\end{align*}
	This is  \eqref{dssert} so that the proof  of Theorem \ref{thm:apriori} is completed.
\end{proof}

 \section{The 3D hyperbolic Prandtl equation}\label{sec:3D}

 The discussion on the 3D hyperbolic Prandtl model is  similar to that of the 2D case with slight modifications on auxiliary functions. Precisely,
  we will use vector-valued auxiliary functions instead of the scalar ones used in the previous sections.   We  denote by $\vec{u}=(u_{1},u_{2})$ and $v$  the tangential and normal velocities respectively,   and   by $(x,y) $ the spatial variables in $\mathbb R^2\times \mathbb R_+ $ with $x=(x_{1},x_{2})$.    As the counterparts   of  the auxiliary functions defined in Section \ref{sec:aux}, we set
  \begin{align*}
  	\vec{  \varphi}=	\partial_t\vec u+(\vec u\cdot\partial_x) \vec u+  v\partial_y \vec u  \  \textrm{  with  }  \    v(t, x,y)=-\int_0^y \partial_x \cdot \vec u(t,x,\tilde y) d\tilde y.
  \end{align*}
  Moreover, we define
    $ \vec{ \mathcal U}=(\mathcal U_1,\mathcal U_2)$ and     $\vec{  \lambda } =(\lambda_1,\lambda_2,\lambda_3,\lambda_4) $  as follows. Let $\mathcal U_j, j=1,2,$ solve the Cauchy problem
 \begin{eqnarray*}
 \left\{
 \begin{aligned}
 & \big(\partial_t+ \vec{u}\cdot \partial_x  + v\partial_y \big)    \int_0^y\mathcal U_j(t,x,\tilde y)  d\tilde y  =  -\partial_{x_j}  v,\\
 & \mathcal U_j|_{t=0}=0.
 \end{aligned}
 \right.
 \end{eqnarray*}
 Accordingly,
 set
 \begin{eqnarray*}
 	\left\{
 	\begin{aligned}
 		\lambda_1&=\partial_{x_1}u_1- (\partial_yu_1)\int_0^y\mathcal U_1 (t,x,\tilde y) d\tilde y,\quad
 		\lambda_2=\partial_{x_2}u_1- (\partial_yu_1)\int_0^y\mathcal U_2 (t,x,\tilde y) d\tilde y,\\
 		\lambda_3&=\partial_{x_1}u_2-  (\partial_yu_2)\int_0^y\mathcal U_1 (t,x,\tilde y) d\tilde y,\quad  \lambda_4=\partial_{x_2}u_2-  (\partial_yu_2)\int_0^y\mathcal U_2(t,x,\tilde y)  d\tilde y.
 	\end{aligned}
 	\right.
 \end{eqnarray*}	
 We also  denote  that
 $$ \vec{ b} =( \vec{u}, \vec{  \mathcal U}, \vec{\lambda}, \vec\varphi),$$
 and  define $|\vec{ b}|_{X_\rho} $ and $|\vec{ b}|_{Y_\rho}$ as in Definition \ref{defabnorm}.
 Then the {\it  a priori} estimate  in Theorem \ref{thm:apriori}  also holds with $\vec a$ replaced by   $ \vec{ b} $.  This can be derived in the same way for the  2D case.  For brevity,   we omit the details.

\appendix

\section{Proof of some estimates}\label{sec:app}

Finally, we present the proof of some estimates used in the previous sections.

\begin{proof}
	[Proof of \eqref{fe1}] For any  $0\leq j\leq [m/2]$, we have that  $m-j\approx m$ so that 
		\begin{align*}
&\frac{m! }{ j!(m-j)!} \frac{ N_{\rho,m+1}  }{ N_{\rho, j+3} N_{m-j+1,\rho}}\\
 & =
\frac{m!}{j!(m-j)!}
\times\frac{\rho^{m+2}(m+2)^9}{[(m+1)!]^{3/2}}\times
\frac{[(j+3)!]^{3/2}}{\rho^{j+4}(j+4)^9}\times\frac{[(m-j+1)!]^{3/2}}{\rho^{m-j+2}(m-j+2)^9}\\
&\lesssim    j! ^{1/2}[(m-j)!]^{1/2}  \times \frac{ j^{9/2}(m-j+1)^{3/2}}{\rho^4 m!^{1/2} (m+1)^{3/2} ( j+4)^9}\lesssim \frac{1}{j+1},
\end{align*}
where we have used the fact that $p!q!\leq (p+q)!$ and  \eqref{loup} in the last inequality.
This gives estimate
\eqref{fe1}.
\end{proof}

 \begin{proof}[Proof of \eqref{fe2}] For any $j$  with $ [m/2]+1\leq j\leq m,$ we have
\begin{align*}
	&\frac{m!}{j!(m-j)!} \frac{N_{\rho, m+1}}{N_{\rho,j+1} N_{\rho,m-j+3} }\\
 & =
\frac{m!}{j!(m-j)!}
\times\frac{\rho^{m+2}(m+2)^9}{[(m+1)!]^{3/2}}\times
\frac{[(j+1)!]^{3/2}}{\rho^{j+2}(j+2)^9}\times\frac{[(m-j+3)!]^{3/2}}{\rho^{m-j+4}(m-j+4)^9}\\
&\lesssim    j! ^{1/2}[(m-j)!]^{1/2}  \times \frac{j^{3/2}(m-j+1)^{9/2} }{\rho^4 m!^{1/2} (m+1)^{3/2} (m- j+4)^9}\lesssim \frac{1}{m-j+1}.\end{align*}
This gives \eqref{fe2}.	
\end{proof}

\begin{proof}[Proof of \eqref{fe3}] For $1\leq i \leq [(k+1)/2]$ we have $k-i\approx k$ so that
 	\begin{align*}
 		& \frac{(k+1)!}{i!(k+1-i)!}  \frac{(k+1)^{-1/2} L_{\rho ,k} }{H_{\rho,4 ,i-1}L_{\rho ,k+1-i}  } \\
 		&\lesssim \frac{(k+1)!}{i!(k+1-i)!}  (k+1)^{-{1\over2}}\times\frac{\rho^{k+2}(k+2)^9}{(k+1)!}\times
\frac{ (i+3)! }{\rho^{i+4}(i+4)^9}\times\frac{ (k+2-i)! }{\rho^{k+3-i}(k+3-i)^9}\\
 		&\lesssim  \frac{ (k+1)^{- 1/2} (k+2-i)}{  (i+1)^{6}}\lesssim  \frac{ (k+2-i)^{ 1/2}}{ i+1}. 	\end{align*}
 		Hence, \eqref{fe3} holds.
 \end{proof}

\begin{proof}[Proof of \eqref{fe4}]
	We use the fact that  $[(k+1)/2]+1 \leq i\leq k+1$  to have
	\begin{align*}
		&\frac{(k+1)!}{i!(k+1-i)!}  \frac{(k+1)^{-1/2} L_{\rho ,k} }{H_{\rho,2, i-2}H_{\rho,3 ,k+2-i}  }\\
		&\lesssim \frac{(k+1)!}{i!(k+1-i)!}  (k+1)^{-{1\over2}}\times\frac{\rho^{k+2}(k+2)^9}{(k+1)!}\times
\frac{ i! }{\rho^{i+1}(i+1)^9}\times\frac{ (k+5-i)! }{\rho^{k+6-i}(k+6-i)^9}\\
 		&\lesssim  \frac{ 1}{  (k+5-i)^5}\lesssim  \frac{ 1}{  k+3-i },
	\end{align*}
	which is \eqref{fe4}.
\end{proof}

 \begin{proof}
	[Proof of \eqref{fe6}] For any $1\leq j\leq [m/2]$ so that
  $m-j\approx m$,  and thus,
\begin{align*}
  &\frac{m!}{j!(m-j)!}    \frac{(m+1)^{1/2} N_{\rho ,m+1} }{ N_{\rho , j+3}H_{\rho ,m-j+1,1}  }\\
  &=
\frac{m! (m+1)^{1\over 2}}{j!(m-j)!}
\times\frac{\rho^{m+2}(m+2)^9}{[(m+1)!]^{3/2}}\times
\frac{[(j+3)!]^{3/2}}{\rho^{j+4}(j+4)^9}\times\frac{(m-j+2)![(m-j+1)!]^{1\over2}}{\rho^{m-j+3}(m-j+3)^9}\\
&\lesssim      j! ^{1/2}[(m-j)!]^{1/2} (m+1)^{1/2} \times \frac{j^{9/2}(m-j+1)^{5/2} }{ m!^{1/2} (m+1)^{3/2} (j+4)^9}\lesssim \frac{(m-j+1)^{3/2}}{j+1}.
\end{align*}
\end{proof}

\begin{proof}
	[Proof of \eqref{fe5}]
	For $[m/2]+1\leq j\leq m$,  we have
	\begin{align*}
&	\frac{m!}{j!(m-j)!}   \frac{  (m+1)^{1/2}  N_{\rho ,m+1} }{ N_{\rho ,j+1}H_{\rho ,m-j+3,1}  } \\
& =
\frac{m! (m+1)^{1\over2}}{j!(m-j)!}
\times\frac{\rho^{m+2}(m+2)^9}{[(m+1)!]^{3/2}}\times
\frac{[(j+1)!]^{3\over2}}{\rho^{j+2}(j+2)^9}\times\frac{(m-j+4)![(m-j+3)!]^{1\over2}}{\rho^{m-j+5}(m-j+5)^9}\\
&\lesssim     j! ^{1/2}[(m-j)!]^{1/2}(m+1)^{1/2}  \times \frac{j^{3/2}(m-j+1)^{11/2} }{  m!^{1/2} (m+1)^{3/2} (m- j+5)^9}\\
&\lesssim  \frac{(j+1)^{1/2}}{m-j+1}.
\end{align*}
This gives \eqref{fe5}. 
\end{proof}

\begin{proof}
	[Proof of \eqref{fe10}] For $1\leq i+j\leq [(m+k+1)/2]$,  we have $m+k-i-j\approx m+k$ 
	so that 
	\begin{align*}
	&{m\choose j} {{k+1}\choose i} (m+k+1)^{-{1\over 2}}(m+1) \frac{ H_{\rho , m+1,k}}{H_{\rho, j+4,i-1}H_{\rho, m-j+1,k+1-i}}\\
	&=\frac{m!}{j!(m-j)!}\frac{(k+1)!}{i!(k+1-i)!}(m+k+1)^{-{1\over 2}}(m+1) \times\frac{\rho^{m+k+2}(m+k+2)^9}{(m+k+1)! [(m+1)!]^{1/2}}\\
	&\qquad\qquad\quad \times \frac{(i+j+3)! [(j+4)!]^{1/2}}{\rho^{i+j+4}(i+j+4)^9} \times \frac{(m+k-i-j+2)! [(m-j+1)!]^{1/2}}{\rho^{m+k-i-j+3}(m+k-i-j+3)^9}\\
	&\lesssim  \frac{(m!)^{1/2}}{ (j!)^{1/2}[(m-j)!]^{1/2}}   \frac{(k+1)!}{i!(k+1-i)!} \times \frac{(i+j+3)! (m+k-i-j+2)! }{ (m+k+1)!}\\
	&\qquad \times (m+k+1)^{-{1\over 2}}(m+1) ^{1/2} \frac{(j+1)^2(m-j+1)^{1/2} }{(i+j+4)^9 } .
\end{align*}
Moreover, by using 
\begin{align*}
	\forall\ \alpha,\beta\in\mathbb Z_+^2 \ \textrm{ with }\ \beta\leq\alpha, \quad {\alpha\choose\beta}\leq {{\abs\alpha}\choose{\abs\beta}},
\end{align*}
we have
\begin{multline*}
	\frac{(m!)^{1/2}}{ (j!)^{1/2}[(m-j)!]^{1/2}}   \frac{(k+1)!}{i!(k+1-i)!}\\
	\leq \frac{ m! }{  j!   (m-j)! }   \frac{(k+1)!}{i!(k+1-i)!}\leq \frac{(m+k+1)!}{(i+j)!(m+k+1-i-j)!}.
\end{multline*}
We then combine the above estimates and observe that  $m+k-i-j\approx m+k $ for $1\leq i+j\leq [(m+k+1)/2]$   to have
\begin{align*}
	&{m\choose j} {{k+1}\choose i} (m+k+1)^{-{1\over 2}}(m+1) \frac{ H_{\rho , m+1,k}}{H_{\rho, j+4,i-1}H_{\rho, m-j+1,k+1-i}}  \\
	&\lesssim     (m+k+1)^{-{1\over 2}}(m+1) ^{1\over2}\frac{(j+1)^2(m-j+1)^{1\over2} (i+j+1)^3 (m+k-i-j+2)}{(i+j+4)^9 } \\
	&\lesssim     (m+1) ^{1/2}(m-j+1)^{1/2}   \frac{1}{ (i+j+3)^4}    (m+k-i-j+2)^{1/2} \\
	&\lesssim   ( j+4) (m-j+1)\frac{1}{ (i+j+1)^2}    (m+k-i-j+2)^{1/2}.
\end{align*}
The proof of \eqref{fe10} is  completed.
\end{proof}

\begin{proof}
	[Proof of \eqref{laetimate}] For any   $[(m+k+1)/2]\leq i+j\leq m+k+1$, following a similar argument as above,  we have
		\begin{align*}
		&{m\choose j} {{k+1}\choose i}  \frac{(m+k+1)^{-1/2}(m+1) H_{\rho , m+1,k}}{H_{\rho, j+2,i-2}H_{\rho, m-j+3,k+2-i}}\\
		&=\frac{m!}{j!(m-j)!}\frac{(k+1)!}{i!(k+1-i)!}(m+k+1)^{-{1\over 2}}(m+1) \times\frac{\rho^{m+k+2}(m+k+2)^9}{(m+k+1)! [(m+1)!]^{1/2}}\\
	&\qquad\qquad\quad \times \frac{(i+j)! [(j+2)!]^{1/2}}{\rho^{i+j+1}(i+j+1)^9} \times \frac{(m+k-i-j+5)! [(m-j+3)!]^{1/2}}{\rho^{m+k-i-j+6}(m+k-i-j+6)^9}\\
	&\lesssim  \frac{(m!)^{1/2}}{ (j!)^{1/2}[(m-j)!]^{1/2}}   \frac{(k+1)!}{i!(k+1-i)!} \times \frac{(i+j)! (m+k-i-j+5)! }{ (m+k+1)!}\\
	&\qquad \times (m+k+1)^{-{1\over 2}}(m+1) ^{1/2} \frac{(j+1) (m-j+1)^{3/2} }{(m+k-i-j+6)^9 }\\
	&\lesssim    (m+k+1)^{-{1\over 2}}(m+1) ^{1/2} \frac{(j+1) (m-j+1)^{3/2} }{(m+k-i-j+6)^5 }\\
	&\lesssim    \frac{(j+1)  }{(m+k-i-j+2)^2 }.
	\end{align*}
	Hence, \eqref{laetimate} follows.
\end{proof}

\begin{proof}
	[Proof of \eqref{init}] By \eqref{defX}, \eqref{mau} and \eqref{ma}, we have
	\begin{align*}
			& |\vec a(0)|_{X_{\rho_0}}^2 =   \sum_{ m=0}^{\infty}    (m+1)     N_{\rho_0,m+1}^2    \norm{\comi y^{\ell-1}\partial_x^{m+1} u_0}_{L^2}^2  \\
	 & \qquad +  \sum_{m,  k\geq 0}       (m+1)^2 H_{\rho_0,m+1,k} ^2\inner{ \norm{  \comi y^\ell  \partial_x^m\partial_y^k\varphi_0 }_{L^2}^2+\norm{  \comi y^\ell  \partial_x^m\partial_y^{k}\partial_y u_0}_{L^2}^2}.\end{align*}
	 This together with 
	 \begin{align*}
	 	\sum_{m,  k\geq 0}       (m+1)^2  H _{\rho_0,m+1 ,k} ^2 \norm{  \comi y^\ell  \partial_x^m\partial_y^{k}\partial_y u_0}_{L^2}^2 \leq \norm{u_0}_{G^{3/2,1}_{2\rho_0,\ell}}^2,
	 \end{align*}
	 and 
	 \begin{align*}
	 	& \sum_{ m=0}^{\infty}    (m+1)     N_{\rho_0,m+1}^2    \norm{\comi y^{\ell-1}\partial_x^{m+1} u_0}_{L^2}^2\\
	 	&\qquad=\sum_{ m=0}^{\infty}    (m+1)\Big(\frac{1}{2} \Big)^{m+2}    N_{2\rho_0,m+1}^2    \norm{\comi y^{\ell-1}\partial_x^{m+1} u_0}_{L^2}^2 \leq 2\norm{u_0}_{G^{3/2,1}_{2\rho_0,\ell}}^2,
	 \end{align*}
	yields
	 \begin{equation}\label{ini+}
			  |\vec a(0)|_{X_{\rho_0}}^2 \leq 3\norm{u_0}_{G^{3/2,1}_{2\rho_0,\ell}}^2
	  +  \sum_{m,  k\geq 0}       (m+1)^2 H_{\rho_0,m+1,k} ^2  \norm{  \comi y^\ell  \partial_x^m\partial_y^k\varphi_0 }_{L^2}^2.
	  \end{equation}
It follows from \eqref{vazero} and $\ell\geq2$ that
\begin{align*}
	& \norm{\comi y^\ell  \partial_x^m\partial_y^k\varphi_0 }_{L^2}\leq  \norm{  \comi y^\ell \partial_x^m\partial_y^ku_1 }_{L^2}\\
	&\qquad +C\sum_{j\leq m,\ i\leq k}{m\choose j}{k\choose i}\norm{\comi y^{\ell-1}\partial_x^{m+1-j}\partial_y^{k-i}u_0}_{L^2}\\
	&\qquad\qquad\qquad\qquad\qquad \times\Big(\norm{\comi y^{\ell-1}\partial_x^j\partial_y^i  u_0}_{L^\infty}+\norm{\comi y^\ell\partial_x^j\partial_y^i \partial_yu_0}_{L_x^\infty L_y^2}\Big).
\end{align*}
	 Hence, by using  Young's inequality
	\eqref{dis}  for   discrete convolution and the fact that $kr^k\leq (1-r)^{-1}$ for any pair $(r,k)\in [0,1/2]\times\mathbb Z_+$, we conclude that
	 \begin{align*}
	 	 \sum_{m,  k\geq 0}       (m+1)^2 H_{\rho_0,m+1 ,k} ^2  \norm{  \comi y^\ell  \partial_x^m\partial_y^k\varphi_0 }_{L^2}^2\leq C\Big(\norm{u_1}_{G^{3/2,1}_{2\rho_0, \ell+1}}^2+ \norm{u_0}_{G^{3/2,1}_{2\rho_0, \ell}}^2\Big).
	 \end{align*}
	 This and \eqref{ini+} imply \eqref{init}. 
\end{proof}

{\bf Acknowledgments.}W.-X. Li is supported by the National Natural Science Foundation of China (Nos. 11961160716, 12325108, 12131017, 12221001) and  the Natural Science Foundation of Hubei Province (No. 2019CFA007). 
T.  Yang is supported by the startup fund of the Hong Kong Polytechnic University P0043962 and the Research Centre for Nonlinear Analysis.
P. Zhang is partially  supported by National Key R$\&$D Program of China under grant
  2021YFA1000800, K. C. Wong Education Foundation and
  by National Natural Science Foundation of China  under Grant 12031006.

\end{document}